\documentclass[times,twoside]{article}

\usepackage[margin=3.3cm]{geometry}
\usepackage{moreverb}

\usepackage[colorlinks,bookmarksopen,bookmarksnumbered,citecolor=red,urlcolor=red]{hyperref}

\ifpdf
\hypersetup{
  pdftitle={Convergence Analysis Leading to a New Deflation Space},
  pdfauthor={N. Spillane, D.J. Szyld}
}
\fi

\usepackage{amsthm}
\usepackage{amsmath}
\usepackage{amssymb}
\usepackage{tikz,pgfplots}
\usepackage{subcaption}
\usepackage{multirow}
\usepackage{enumerate}
\usepackage{enumitem}
\usepackage{stmaryrd}

\def\nref#1{{\rm (\ref{#1})}}
\def\cvd{~\vbox{\hrule\hbox{%
   \vrule height1.3ex\hskip1.3ex\vrule}\hrule } }

\theoremstyle{plain}
\newtheorem{theorem}{Theorem}

\newtheorem{definition}{Definition}

\newtheorem{remark}{Remark} 
\newtheorem{lemma}{Lemma} 

\newcommand{\HDD}{\ensuremath{\mathbf{H}_{\text{DD}}}}
\newcommand{\rhs}{\ensuremath{\mathbf{b}}}
\newcommand{\Minv}{{\ensuremath{{\mathbf{M}}^{-1}}}}
\newcommand{\matid}{\ensuremath{\mathbf{{I}}}}
\newcommand{\bu}{\ensuremath{\mathbf{u}}}
\newcommand{\by}{\ensuremath{\mathbf{y}}}
\newcommand{\bR}{\ensuremath{\mathbf{R}}}
\newcommand{\bP}{\ensuremath{\mathbf{P}}}

\newcommand{\bD}{\ensuremath{\mathbf{D}}}
\newcommand{\bM}{\ensuremath{\mathbf{M}}}
\newcommand{\bN}{\ensuremath{\mathbf{N}}}
\newcommand{\bW}{\ensuremath{\mathbf{W}}}
\newcommand{\bZ}{\ensuremath{\mathbf{Z}}}
\newcommand{\bU}{\ensuremath{\mathbf{U}}}
\newcommand{\bV}{\ensuremath{\mathbf{V}}}
\newcommand{\bQ}{\ensuremath{\mathbf{Q}}}
\newcommand{\bI}{\ensuremath{\mathbf{I}}}
\newcommand{\bY}{\ensuremath{\mathbf{Y}}}

\newcommand{\bPi}{\ensuremath{\boldsymbol{\Pi}}}

\newcommand{\bx}{\ensuremath{\mathbf{x}}}
\newcommand{\bb}{\ensuremath{\mathbf{b}}}

\newcommand{\bA}{\ensuremath{\mathbf{A}}}

\newcommand{\bH}{\ensuremath{\mathbf{H}}}
\newcommand{\br}{\ensuremath{\mathbf{r}}}

\newcommand{\bz}{\ensuremath{\mathbf{z}}}

\newcommand{\spann}{\ensuremath{\operatorname{span}}}
\newcommand{\range}{\ensuremath{\operatorname{range}}}

\newcommand{\green}[1]{{\color{green!65!black} #1}}

\usepackage{stmaryrd}

\usetikzlibrary{fit,calc}

\title{New Convergence Analysis of GMRES with \\ Weighted Norms, Preconditioning and Deflation,
\\ Leading to a New Deflation Space\thanks{This version dated \today}}

\author{Nicole Spillane\thanks{CNRS, CMAP, \'Ecole polytechnique, Institut Polytechnique de Paris, 91128 Palaiseau Cedex, France (\textit{nicole.spillane@cmap.polytechnique.fr})}   
~and Daniel B. Szyld\thanks{Department of Mathematics, Temple University, Philadelphia, PA 19122, USA (\textit{szyld@temple.edu})} 
}

\begin{document}

\date{}

\maketitle



\pagestyle{myheadings}
\markboth{Nicole SPILLANE and Daniel B. SZYLD}{Convergence Analysis Leading to a New Deflation Space}

\abstract{
New convergence bounds are presented for weighted, preconditioned, and deflated GMRES
for the solution of large, sparse, non-Hermitian linear systems.
These bounds are given for the case when the Hermitian
part of the coefficient matrix is positive definite, the preconditioner is Hermitian positive definite, and
the weight is equal to the preconditioner.
The new bounds are a novel contribution
in and of themselves. In addition, they are sufficiently explicit to indicate how to choose the preconditioner and the
deflation space to accelerate the convergence. One such choice of deflating space is presented,
and numerical experiments illustrate the effectiveness of such space.
}

\tableofcontents

\section{Introduction}

Our aim in this paper is to study effective solutions of linear systems of the form
\begin{equation} \label{linsys:eq}
\bA \bx = \mathbf b, \quad \bA \in \mathbb K^{n \times n},
\end{equation}
where $\mathbb K = \mathbb R$ or $\mathbb C$ and $\bA$ is a large sparse  nonsingular matrix. 
Particular emphasis will be on the cases where $\bA$ has a
positive definite Hermitian (hpd) part. We refer to such matrices $\bA$  as positive definite (pd) matrices. 

We are interested in studying effective approaches to accelerate the convergence of the
well-known and widely used GMRES method 
\cite{zbMATH03967793}
for the solution of linear systems. 
There are essentially three components for a successful strategy for this accelerations, which can
be used alone or be combined:
\begin{itemize}
\item preconditioning,
\item weighting,
\item deflating.
\end{itemize}
Standard references for preconditioning include
\cite{be2002,zbMATH01953444,Simoncini.Szyld.07}; for weighted GMRES
\cite{emn2016, zbMATH01268410,  gp2014}; and for deflation
\cite{bepw1998,cs1997,en2009,ggln2013}, and more recently \cite{zbMATH07152809}.
See also \cite{WDGMRES_Sylvester}, where a weighted and deflated version of GMRES is applied to
Sylvester matrix equations.
There are some equivalences between weighting and preconditioning, see, e.g., \cite[Chapter 4]{Malek_Strakos.book}, but here we keep the weight
and the preconditioner separate for further flexibility and generality.
We denote the weighted preconditioned and deflated GMRES algorithm as WPD-GMRES,
and it corresponds to the case where all three acceleration tools are used. 

Our objective in this paper is to propose a new convergence bound for WPD-GMRES that 
is sufficiently explicit to indicate how to choose the preconditioner, weight matrix, 
and especially deflation spaces. 
The new results generalize those in \cite{spillane2023hermitian} where deflation was not considered. Here also, special emphasis is on Hermitian preconditioning and on applying WPD-GMRES in the preconditioner norm, as was done, e.g., in \cite{zbMATH01201042,zbMATH01096035}. 
For the cases we consider, the new bounds improve upon the more general bound of Elman
\cite{EES.83,elman1982iterative}.

In particular, in Section~\ref{hpd_precond:sec}, we present a result explicitly giving
conditions on the preconditioner and the deflation spaces so as to assure fast convergence.
Then, in Section~\ref{sec:Newspace}, we propose a new deflation space which is inspired
by this new bound.
Numerical experiments in Section~\ref{sec:Numerical} illustrate the new results with
the choice of the new space, and show its effectiveness.
In fact, in one of the example problems we present, GMRES fails to converge, while
a deflation space with dimension less than 10\% of the size of the problems,
allows GMRES to converge.

In part inspired by the success of the GenEO coarse 
space \cite{2011SpillaneCR,spillane2013abstract}, and as already mentioned, by the
new bounds we obtain, we use as a deflation space, 
appropriately chosen eigenvectors of the generalized eigenvalue problem 
$\bN \bz = \lambda \bM \bz$, where $\bM$ is the Hermitian part of $\bA$, which is
assumed to be positive definite, and $\bN$ is the skew-Hermitian part of $\bA$.

In summary, our contribution consists of proving a new convergence bound for weighted
preconditioned GMRES, for the special case that the (right) preconditioneer is the same
as the weight, and $A$ is positive definite; and to provide a special deflating space which allows
us to quantify this bound explicitly.

\section{Preliminaries}
We begin by stating some results for weighted GMRES for singular systems. 
As we describe in the next section, deflating produces a 
consistent singular system and thus, analyzing the singular case will be useful
for our analysis of deflated GMRES.

Weighted GMRES is the version of GMRES in which a general inner product 
$\langle \cdot, \cdot \rangle_\bW$ replaces the Euclidean inner product \cite{zbMATH01268410}. 
The Hermitian positive definite (hpd) matrix 
$\bW$ such that $\langle \bx, \by \rangle_\bW = \langle \bW \bx, \by \rangle$ will 
be referred to as the weight matrix. The user inputs an initial vector $\bx_0 \in \mathbb K^n$. 
The approximate solution at iteration $i$ is then characterized by
\begin{equation}
\label{eq:GMRESmin}
\| \br_i \|_\bW = \operatorname{min}\left\{ \| \mathbf b - \bA \bx \|_{ \bW } ; \, {\bx \in \bx_{0}  + \mathcal K_i} \right\},
\end{equation}
where $\mathcal K_i$ is  the Krylov subspace 
\begin{equation}
\label{eq:AKrylov}
\mathcal K_i = \mathcal K_i(\br_0, \bA) =  \spann \{\br_0, \bA \br_0 , \dots , \bA^{i-1}\br_0 \}; \quad \br_0 =  \mathbf b - \bA \bx_0.
\end{equation}

GMRES for singular systems is studied, e.g, in \cite{zbMATH00994071} and \cite{zbMATH06043449} (where GCR is also considered). A very useful result is recalled in Theorem~\ref{th:singularGMRES} with a straightforward generalization to weighted GMRES. 
The proof of \cite[Theorem 2.6]{zbMATH00994071} applies here with the change of inner product;
see also~\cite{ggln2013}.

\begin{theorem}
\label{th:singularGMRES}
Suppose that $\operatorname{range}(\bA) \cap \operatorname{ker}(\bA) = \{0\}$. If $\bA \bx = \rhs$ is consistent, \textit{i.e.}, if it  admits a solution, then, in exact arithmetic,  weighted GMRES determines a solution without breakdown at some step and breaks down at the next step through degeneracy of the Krylov subspace. 
\end{theorem}

The takeaway from the theorem is that for consistent linear systems, under the condition that 
$\operatorname{range}(\bA) \cap \operatorname{ker}(\bA) = \{\mathbf{0}\}$, we can proceed through the iterations 
in the same manner as with nonsingular systems. 
In exact arithmetic, the characterization \eqref{eq:GMRESmin} of the iterate $\bx_i$ remains valid  in the singular case, until the algorithm breaks down, at 
which point the exact solution has been found.  

We continue in this preliminaries section by reviewing properties of the generalized eigenvalue problem we use
for our new deflation space. For completeness we give the proof of the following result.


\begin{lemma}
\label{lem:prelim}
Let us assume that $\bM$ and $\bN$ are two $n \times n$ matrices with the further assumption that $\bM$ is hpd and $\bN$ is skew-Hermitian. Consider the generalized eigenvalue problem for matrix pencil $(\bN, \bM)$: 
find $\lambda_{j} \in \mathbb C$ and $\bz^{(j)} \in \mathbb C^n \setminus \{\mathbf{0}\}$ such that
\begin{equation}
\label{eq:gevpNM}
\bN \bz^{(j)} = \lambda_{j} \bM \bz^{(j)}.
\end{equation}
Then, the eigenvectors $\bz^{(j)}$ can be chosen to form an $\bM$-orthonormal basis of $\mathbb C^n$, and
the eigenvalues $\lambda_{j}$ are either $0$ or purely imaginary. 
\end{lemma}

\begin{proof}
We first prove that the eigenvectors can be chosen to form an $\bM$-orthonormal basis of $\mathbb C^n$. 
Let $(\lambda_{j}, \bz^{(j)})$ denote an eigenpair of the generalized eigenvalue problem \eqref{eq:gevpNM}. 
It is immediate to observe that an equivalent eigenvalue problem is 
\[
\bM^{-1/2} \bN \bM^{-1/2} \tilde \bz^{(j)} = \lambda_{j}  \tilde \bz^{(j)}; \quad \tilde \bz^{(j)}  =  \bM^{1/2} \bz^{(j)}  ,
\] 
where $\bM^{1/2}$ denotes the matrix square root of $\bM$.\footnote{By  \cite[theorem 7.2.6, page 439]{hornjoh:85}, $\bM^{1/2}$ is well defined as the unique Hermitian positive semi-definite matrix such that $(\bM^{1/2})^2 = \bM$ and moreover $\bM^{1/2}$ is positive-definite because $\bM$ is positive-definite.}

Matrix $\bM^{-1/2} \bN \bM^{-1/2} $ is skew-Hermitian: $(\bM^{-1/2} \bN \bM^{-1/2})^* = (\bM^{-1/2})^* \bN^* (\bM^{-1/2})^* =  - \bM^{-1/2} \bN \bM^{-1/2} $. Consequently $\bM^{-1/2} \bN \bM^{-1/2} $ is normal and the spectral theorem states that it is unitarily diagonalizable:
\[
\bM^{-1/2} \bN \bM^{-1/2} = \bU \bD \bU^*, \bD \text{ diagonal}, \, \bU \text{ unitary} \text{ (\textit{i.e.}, } \bU^* \bU = \matid).
\] 
It immediately follows, by setting $\bV = \bM^{-1/2} \bU$ that 
\[
\bV^* \bN \bV = \bD  , \, \bD \text{ diagonal}, \, \bV \text{ satisfies } \bV^* \bM \bV = \matid,
\mbox{\rm ~and that~}
\bN \bV = \bM \bV \bD, 
\]
which is equivalent to 
\[
\bN \bz^{(j)} = \lambda_j \bM \bz^{(j)}, \, \forall j, j=1,\ldots, n\,  \text{ where } \bz^{(j)}, 
\text{ is the $j$-th column of } \bV \text{ and } \lambda^{(j)} = D_{jj}.
\] 
Thus,
the eigenvectors in generalized eigenvalue problem \eqref{eq:gevpNM} can be chosen to form an $\bM$-orthonormal basis of $\mathbb C^n$.

Next we prove that the non-zero eigenvalues are purely imaginary. 
Let $(\lambda_{k}, \bz^{(k)})$  denote any eigenpair of the generalized eigenvalue problem \eqref{eq:gevpNM} then 
\[
\langle \bN \bz^{(k)}, \bz^{(k)} \rangle = \lambda_{k} \langle \bM \bz^{(k)}, \bz^{(k)} \rangle = 
\langle \bz^{(k)}, \bN^* \bz^{(k)} \rangle = - \langle \bz^{(k)}, \bN \bz^{(k)} \rangle = - {\lambda_{k}}^* \langle \bM \bz^{(k)}, \bz^{(k)} \rangle.
\]
Since $\bz^{(k)}$ is an eigenvector, $\bz^{(k)}$ is non-zero. Consequently $ \lambda_{k} \langle \bM \bz^{(k)}, \bz^{(k)} \rangle =- {\lambda_{k}}^* \langle \bM \bz^{(k)}, \bz^{(k)} \rangle$ implies that $\lambda_{k} + {\lambda_{k}}^* = 0 =2 \Re(\lambda_{k})$.
\end{proof}

Next we set $\bM$ and $\bN$ to be respectively the Hermitian and skew-Hermitian parts of $\bA$:
\begin{equation}
\label{eq:splitA}
\bA = \bM + \bN, \, \bM = \frac{\bA + \bA^*}{2} \text{ and } \bN = \frac{\bA - \bA^*}{2} \cdot
\end{equation}
We prove a few straightforward properties of the eigenpairs, which we will use later in the paper. If $(\lambda_{k}, \bz^{(k)})$  denotes an eigenpair of the generalized eigenvalue problem \eqref{eq:gevpNM} then
\[
\bA \bz^{(k)} = (\bM + \bN) \bz^{(k)} = (1 + \lambda_{k})\bM \bz^{(k)}, 
\]
where $(1 + \lambda_{k}) \neq 0$ because $\Re(\lambda^{(k)}) = 0$. Similarly, $\bA^* \bz^{(k)} = (1 - \lambda_{k})\bM \bz^{(k)}$ with $(1 - \lambda_{k}) \neq 0$. A consequence is that 
\[
\operatorname{span}(\bA \bz^{(k)}) = \operatorname{span}(\bM \bz^{(k)}) = \operatorname{span}(\bA^* \bz^{(k)}) .
\] 
Since $\bM$ is invertible it also holds that 
\[
\bM^{-1} \bN \bz^{(k)} = \lambda_{k} \bz^{(k)}; \quad (\matid + \bM^{-1} \bN) \bz^{(k)} = (1 + \lambda_{k}) \bz^{(k)};  \quad (\matid + \bM^{-1} \bN)^{-1} \bz^{(k)} = (1 + \lambda_{k})^{-1} \bz^{(k)}. 
\]

\subsection{Weighted and Deflated GMRES}

The purpose of deflation is to replace the linear system \nref{linsys:eq} by a projected linear system that is 
easier to solve iteratively. The deflation operators are introduced next. 

\begin{definition}
\label{def:PDQD}
Let $\bY, \bZ \in \mathbb K^{n\times m}$ be two full rank matrices. Under the assumption that $\ker(\bY^*) \cap \range(\bA \bZ) = \{ \mathbf{0}\}$, let
\[
\bP_D := \matid - \bA \bZ(\bY^* \bA \bZ)^{-1} \bY^* \text{ and }  \bQ_D := \matid - \bZ (\bY^* \bA \bZ)^{-1}\bY^* \bA.
\]
These are projection operators called the deflation operators.  
\end{definition}

The following lemma gives some simple but useful properties of the deflation operators. 

\begin{lemma}
\label{lem:PDQD}
The deflation operators satisfy
\[
\bP_D \bA = \bA \bQ_D = \bP_D \bA \bQ_D,
\]
and
\begin{align*}
\operatorname{ker} (\bP_D) = \operatorname{range}(\bA\bZ), & \quad & \operatorname{range} (\bP_D) = (\operatorname{ker} (\bP_D^*))^\perp = \operatorname{ker}( \bY^*  ), \\
\operatorname{ker} (\bP_D^*) = \operatorname{range} (\bY) ,& \quad  & \operatorname{range} (\bP_D^*) = (\operatorname{ker} (\bP_D))^\perp = \operatorname{ker}( \bZ^*\bA^*  ),  \\
\operatorname{ker} (\bQ_D)  = \operatorname{range}(\bZ), & \quad  & \operatorname{range} (\bQ_D) = (\operatorname{ker} (\bQ_D^*))^\perp = \operatorname{ker}(\bY^*  \bA   ) , \\
\operatorname{ker} (\bQ_D^*) = \operatorname{range} (\bA^*  \bY) , & \quad & \operatorname{range} (\bQ_D^*) = (\operatorname{ker} (\bQ_D))^\perp = \operatorname{ker}( \bZ^*  ).   \\
\end{align*}
\end{lemma}
\vspace*{-5mm}
Let $\bx_*$ be the solution of \nref{linsys:eq}. We write $ \bx_* = \bQ_D \bx_* + (\bI - \bQ_D) \bx_*$, and
we rewrite the linear system \nref{linsys:eq}  as two independent linear systems for each of the two components
as follows,
\begin{align*}
\bA \bx_* = \bb &\Leftrightarrow \left\{\bP_D \bA \bx_* = \bP_D \bb \text{ and } (\bI - \bP_D)\bA \bx_* = (\bI - \bP_D)\bb \right\}\\
&\Leftrightarrow \left\{  \bA \bQ_D \bx_* = \bP_D \bb \text{ and } \bA (\bI - \bQ_D) \bx_* = (\bI - \bP_D)\bb\right\}.
\end{align*}

Each of the two linear systems can be solved by a different linear solver. On one hand, 
\begin{equation}\label{deflatedsys:eq}
(\matid - \bQ_D) \bx_* = \bZ (\bY^* \bA \bZ) ^{-1} \bY^* \rhs
\end{equation}
is computed with a direct solver. On the other hand, $\bQ_D \bx_*$ is computed by applying (preconditioned) weighted GMRES to the consistent, so called \textit{deflated} linear system 
\begin{equation}\label{deflatedsysP:eq}
\bP_D \bA \tilde\bx = \bP_D \rhs 
\end{equation}
and setting $\bQ_D \bx_* = \bQ_D \tilde\bx $. This is justified by the following
one-line  proof (which is essentially \cite[Lemma 3.2]{zbMATH07152809}),
\[
\bA \bQ_D \tilde\bx = \bP_D \bA \tilde\bx = \bP_D \rhs= \bP_D \bA \bx_* =    \bA \bQ_D \bx_* \Leftrightarrow  \bQ_D \tilde\bx =  \bQ_D \bx_*, 
\]
since  $\bA$ is nonsingular; see further \cite{zbMATH07152809, zbMATH05837408} for more details on deflated GMRES.

In those references and in this paper,
it is implicitly assumed that the number of columns $m$ of $\bY$ and $\bZ$ is not too large so that 
the solution
of the solution of a linear system with the $m \times m$ coefficient matrix  $\bY^* \bA \bZ$ is not too expensive. 
Such solutions are needed when
computing $(\matid - \bQ_D) \bx_*$ as in  \nref{deflatedsys:eq}, and at every application of $\bP_D$ and $\bQ_D$. 

We now focus on solving \nref{deflatedsysP:eq}.
If weighted GMRES is applied directly to this projected system, Theorem~\ref{th:singularGMRES} 
tells us that weighted GMRES converges to the solution as long as $\ker(\bP_D \bA) \cap \range(\bP_D \bA) = \{\mathbf 0 \}$, or by 
Lemma~\ref{lem:PDQD}, if  $\range(\bZ) \cap \ker(\bY^*)  = \{\mathbf 0 \}$. 
For the iterative solution of this system \nref{deflatedsysP:eq}, we consider the use of a preconditioner,
as we discuss next.

\section{Weighted and Deflated right-preconditioned GMRES}

Let $\bH$ be a nonsingular matrix in $\mathbb K^{n\times n}$. We will call it the preconditioner. We precondition the deflated system on the right, which means that we solve the following system,
\begin{equation}
\label{eq:PAHu=Pb}
\bP_D \bA \bH  \bu = \bP_D \rhs; \text{ and then set } \tilde\bx = \bH \bu. 
\end{equation}
{In practice the algorithm is implemented in the $\tilde \bx$ variable rather than in the $ \bu$ variable. This is trivial since the $i$-th residual is $ \bP_D \rhs - \bP_D \bA \bH \tilde \bu_i =  \bP_D \rhs - \bP_D \bA  \tilde\bx_i$. The algorithm produces approximate solutions for $\tilde \bx$ that we will denote by $\bx_i$ and that satisfy $\bx_i = \bH \bu_i$. } 

Equation~\eqref{eq:PAHu=Pb} is a consistent linear system with a singular coefficient matrix $ \bP_D \bA \bH$. 
By Theorem~\ref{th:singularGMRES} (see also \cite[Theorem 3.4]{zbMATH07152809}), 
weighted and deflated preconditioned GMRES converges for every starting vector if  
\[
\operatorname{range}(\bP_D \bA \bH) \cap \operatorname{ker}(\bP_D \bA \bH) = \{\mathbf 0\} \, \Leftrightarrow \, \operatorname{range}(\bP_D) \cap  \operatorname{ker}(\bQ_D \bH)= \{\mathbf 0\} ,
\]
since $\bP_D \bA = \bA \bQ_D$. By Lemma~\ref{lem:PDQD}, the condition can be rewritten as 
\[
\operatorname{ker}( \bY^*  ) \cap \operatorname{range}(\bH^{-1} \bZ)= \{ \mathbf{0}\} .
\]
\begin{remark}
This is the same condition as in \cite[Theorem 3.5]{zbMATH07152809} where left preconditioning is considered. 
Indeed, the Krylov subspaces ${\mathcal K_i}(\bP_D \bA \bH, \bP_D \rhs -  \bP_D \bA \bH \bu_0)$ with $\bu_0 = \bH \bx_0 $ and ${\mathcal K_i}(\bH \bP_D \bA , \bH \bP_D \rhs - \bH \bP_D \bA \bx_0)$ stop growing at the same iteration, \textit{i.e.,} 
the coefficient matrices have the same grades in the sense of \cite[Section 6.2]{zbMATH01953444}. Indeed,  
${\mathcal K_i}(\bH \bP_D \bA , \bH \bP_D \rhs - \bH \bP_D \bA \bx_0) = \bH {\mathcal K_i}(\bP_D \bA \bH, \bP_D \rhs -  \bP_D \bA \bH \bu_0)$ with $\bx_0 = \bH \bu_0 $. 
\end{remark}

The following theorem summarizes the two fundamental conditions that we have just identified. 

\begin{theorem}
\label{th:wellposed}
The deflation operators are well defined and weighted GMRES does not break down when solving the deflated and preconditioned linear system \eqref{eq:PAHu=Pb} if 
\[
\ker(\bY^*) \cap \range(\bA \bZ) = \{ \mathbf{0}\} \text{ and }\operatorname{ker}( \bY^*  ) \cap \operatorname{range}(\bH^{-1} \bZ)= \{ \mathbf{0}\}.
\]
\end{theorem}

Two cases stand out that will be useful further on in the article.

\begin{lemma}
\label{lem:Pdorth}
 If $\bH$ is hpd and $\bY = \bH \bA \bZ$, then the projection operator $\bP_D$ is orthogonal in the $\bH$ inner product. Moreover the condition $\operatorname{ker}( \bY^*  ) \cap \operatorname{range}(\bA \bZ)= \{ \mathbf{0}\}$ in Theorem~\ref{th:wellposed} is automatically satisfied.  
\end{lemma}
\begin{proof}
Let us assume that $\bH$ is hpd, then 
\begin{align*} 
\text{$\bP_D$ is $\bH$-orthogonal } & \Leftrightarrow& \ker(\bP_D) \perp^\bH \range(\bP_D) \\
& \Leftrightarrow&  \range(\bA \bZ)  \perp^\bH \ker(\bY^*) \\
& \Leftrightarrow&   \range(\bH \bA \bZ)  = \range(\bY), 
\end{align*}
a condition that is of obviously satisfied if $\bH \bA \bZ  = \bY $. 
It also follows from the assumptions that 
\[
 \ker(\bY^*) = \ker((\bH \bA \bZ)^*) = (\range(\bH \bA \bZ))^\perp =  (\range(\bA \bZ))^{\perp^\bH} \text{ so } \ker(\bY^*) \cap \range(\bA \bZ) = \{ \mathbf{0}\}.  
\]
\end{proof}

The following result gives a condition relating the preconditioner $\bH$ and the choice of the
deflating subspace represented by $\bY$.

\begin{lemma}
\label{lem:invariant} 
If $\bY$ is an invariant subset of $\bH^* \bA^*$, then $\bQ_D \bH \bP_D = \bH \bP_D$. Moreover, the two conditions in Theorem~\ref{th:wellposed} are equivalent.  
\end{lemma}
\begin{proof}
 Let $\bY$ be an invariant subset of $\bH^* \bA^*$. Since $\bQ_D$ is a projection, $\bQ_D \bH \bP_D = \bH \bP_D$ if $\range (\bH \bP_D) = \range ( \bQ_D)$ or equivalently $\ker (\bY^* \bH^{-1}) =  \ker (\bY^* \bA)$ or, again equivalently $\range (\bA^* \bY) = \range (\bH^{-*} \bY)$. The condition holds if $\bY$ is an invariant subset of $\bH^* \bA^*$. Moreover the conditions in  Theorem~\ref{th:wellposed} can be equivalently rewritten as  
\[
\ker(\bY^* \bA) \cap \range(\bZ) = \{ \mathbf{0}\} \text{ and }\operatorname{ker}( \bY^* \bH^{-1}  ) \cap \operatorname{range}(\bZ)= \{ \mathbf{0}\},
\]
showing that they are equivalent when  $\bY$ is an invariant subset of $\bH^* \bA^*$.
\end{proof}

\begin{remark}
The projection operators $\bP_D$ and $\bQ_D$ are entirely defined through their range and their kernel. 
This means that, $\bY$ and $\bZ$ need only be defined up to their ranges, not necessarily for the
particular choice of their columns. 
\end{remark}

\section{General Convergence of WPD-GMRES}
The following result extends \cite[Theorem 3]{spillane2023hermitian}
to the case where we have a deflation space.
\begin{theorem}
\label{th:DRPconv}
Assume that the two conditions from Theorem~\ref{th:wellposed} are satisfied. Let 
\linebreak 
$\theta(\bA, \bH, \bW, \bY, \bZ)$, indexed by the operator $\bA$, the preconditioner $\bH$, the weight matrix $\bW$ as well as the deflation spaces represented by $\bY$ and $\bZ$, be defined by
\begin{equation}
\label{eq:DRPsuff}
\theta(\bA, \bH, \bW, \bY, \bZ) :=  \operatorname{inf}\limits_{\by\in\range (\bP_D)\setminus\{\mathbf 0\}}\frac{|\langle {\bP_D \bA \bH \by}, \by \rangle_\bW|^2}{ \| \bP_D \bA\bH \by \|_\bW^2  \|\by\|_\bW^2} \cdot
\end{equation}
Then, at any iteration of WPD-GMRES (\textit{i.e.}, weighted GMRES applied to \eqref{eq:PAHu=Pb}) the relative residual norm satisfies
\[
\frac{ \|\br_{i} \|_\bW^2}{\|\br_{i-1} \|_\bW^2} \leq  1 - \theta(\bA, \bH, \bW, \bY, \bZ),
\]
where $\br_i =\bP_D \rhs - \bP_D \bA \bx_i   =\bP_D \rhs - \bP_D \bA \bH \bu_i.$
\end{theorem}

Note that by Cauchy-Schwarz inequality for the $\bW$-norm, we have that $\theta(\bA, \bH, \bW, \bY, \bZ) < 1$.

\vspace*{3mm}

\noindent
{\it{Proof of Theorem \ref{th:DRPconv}}}.
By Theorem~\ref{th:wellposed}, there is no breakdown until convergence has been achieved. So at iteration $i$ of 
weighted GMRES applied to \eqref{eq:PAHu=Pb} it holds that 
\begin{equation*}
\| \br_i \|_\bW = \operatorname{min}\left\{ \| \bP_D  \mathbf b -  \bP_D \bA \bH  \bu \|_{ \bW } ; \, {\bu \in \bu_{0}  + \spann \{\br_0,\bP_D  \bA \bH   \br_0 , \dots , (\bP_D\bA  \bH )^{i-1}\br_0 \}   } \right\},
\end{equation*}
where $\br_0 = \bP_D \mathbf b -  \bP_D \bA \bH  \bu_0  = \bP_D \mathbf b -  \bP_D \bA  \bx_0$.
Written in the $\bx$ variable, the minimization result is 
\vspace*{-3mm}
\begin{equation*}
\| \br_i \|_\bW = \operatorname{min}\left\{ \|\bP_D \mathbf b -  \bP_D \bA \bx \|_{ \bW } ; \, {\bx \in \bx_{0}  + \spann \{\bH \br_0, \bH 
\bP_D \bA 
\bH  \br_0 , \dots , (\bH \bP_D\bA  )^{i-1}\bH\br_0 \}   } \right\}.
\end{equation*}
It can be seen that $\bx_{i-1} +\spann( \bH \br_{i-1}) \subset \bx_{0}  + \spann \{\bH \br_0, \bH \bP_D \bA   
\bH
\br_0 , \dots , (\bH\bP_D  \bA )^{i-1} \bH \br_0 \}$ and thus by taking the minimum over a smaller set, the minimum is no smaller, therefore,
\begin{align*}
\| \br_i \|_\bW &\leq \operatorname{min}\left\{ \| \bP_D\mathbf b - \bP_D  \bA \bx \|_{ \bW } ; \, {\bx \in \bx_{i-1}  + \spann (\bH  \br_{i-1}  )   } \right\}\\
& =  \operatorname{min}\left\{ \| \br_{i-1} -  \bP_D \bA \by \|_{ \bW } ; \, {\by \in \spann (\bH  \br_{i-1} )   } \right\}\\
& =  \| \br_{i-1} -  \alpha_{i-1}\bP_D  \bA \bH  \br_{i-1} \|_{ \bW } \text{ with } \alpha_{i-1} = \frac{\langle {\bP_D \bA \bH \br_{i-1}}, \br_{i-1} \rangle_\bW}{ \|\bP_D \bA \bH \br_{i-1} \|_\bW^2 } \cdot
\end{align*}
The value of $\alpha_{i-1}$ comes from projecting $\br_{i-1}$ $\bW$-orthogonally onto $ \spann ( \bP_D \bA \bH \br_{i-1})$. It now holds that $( \br_{i-1} -  \alpha_{i-1}\bP_D  \bA \bH  \br_{i-1}) \perp^\bW \bP_D \bA \bH \br_{i-1}$ and    
\[
\| \br_i \|_\bW ^2 \leq \|\br_{i-1} -  \alpha_{i-1}\bP_D  \bA \bH  \br_{i-1}\|_\bW^2 = \|\br_{i-1}\|_\bW^2 - |\alpha_{i-1}|^2 \|\bP_D  \bA \bH  \br_{i-1} \|_\bW^2 ~.
\]
The result follows by dividing by $\| \br_{i-1} \|_\bW^2$ and recalling that $\br_{i-1} \in \range(\bP_D)$.
\cvd 

\begin{remark}
The convergence bound in Theorem~\ref{th:DRPconv} is pessimistic for GMRES. Indeed, it is derived from $\| \br_i \|_\bW \leq \operatorname{min}\left\{ \|\bP_D \mathbf b - \bP_D  \bA \bx \|_{ \bW } ; \, {\bx \in \bx_{i-1}  + \spann (\bH  \br_{i-1}  )   } \right\}$ where the global minimization property of GMRES has been bounded by minimizing over a one-dimensional space. For this reason, the bound in Theorem~\ref{th:DRPconv} holds also for all restarted and truncated versions of GMRES and even for the minimal residual algorithm. The remark carries over to all convergence results in the article since they are essentially bounds for $\theta(\bA, \bH, \bW, \bY, \bZ)$.  
\end{remark}

For left preconditioning, the same bound holds with the norms on the left hand side replaced by the norms of the preconditioned residuals.  

\section{The case of hpd preconditioning for $\bA$ positive definite}
\label{hpd_precond:sec}
In the remainder of this article  we make the following three assumptions, which are somehow natural to consider.
\begin{itemize}
\item The coefficient  matrix $ \bA$ is positive definite (in the sense that it has positive definite Hermitian part),
\item The preconditioner $\bH$ is hpd,
\item WPD-GMRES is applied using the inner product induced by the preconditioner, \textit{i.e.},  $\bW = \bH$. 
\end{itemize}

The quantity in the convergence bound of Theorem~\ref{th:DRPconv} can now be rewritten as 
\begin{align}
\theta(\bA, \bH, \bH, \bY, \bZ) &=  \operatorname{inf}\limits_{\by\in\range (\bP_D)\setminus\{\mathbf 0\}} \frac{|\langle \bH\bP_D  \bA \bH \by, \by \rangle|^2}{ \|\bP_D  \bA \bH \by \|_{\bH}^2 \| \by\|_{\bH}^2 } 
\label{theta:eq}
\\ 
&=  \operatorname{inf}\limits_{\by\in\range (\bH \bP_D)\setminus\{\mathbf 0\}} \frac{|\langle\bP_D  \bA \by, \by \rangle|^2}{ \| \bP_D  \bA \by \|_{\bH}^2 \| \by\|_{\bH^{-1}}^2 } \cdot \nonumber
\end{align}

From here, two cases are considered that differ by the constraint imposed on the deflation spaces. In each case, the objective is to make explicit a condition that must be satisfied by $\bH$, $\bY$ and $\bZ$ in order to ensure fast convergence. The results are summed up in the next theorem.
The following is our main result which extends \cite[Theorem 6]{spillane2023hermitian}
to the case where we have a deflation space.
\begin{theorem}
\label{th:abscv2}
Let us assume that $\bA$ is positive definite, $\bH$ is hpd and $\bW = \bH$. 
Let
\[
\bM = \frac{\bA + \bA^*}{2}
\]
denote the Hermitian part of $\bA$, and $\lambda_{\min}(\bH\bM)$ and $\lambda_{\max}(\bH\bM)$ denote the extreme eigenvalues of $\bH \bM$. The quantity $\theta$ in the convergence result of WPD-GMRES  (Theorem~\ref{th:DRPconv})  can be bounded as follows. 
\begin{enumerate}
\item If $\bY = \bH \bA \bZ$, \textit{i.e.}, $\bP_D$ is $\bH$-orthogonal then 
\[
\theta(\bA, \bH, \bH, \bH \bA \bZ , \bZ) \geq  \frac{\lambda_{\min}(\bH \bM)}{\lambda_{\max}(\bH\bM)} \times  \operatorname{inf}\limits_{\by\in  \ker(\bZ^* \bA^* \bA^{-1} )\setminus\{\mathbf 0\}} \frac{|\langle  \bA^{-1} \by, \by \rangle|}{ \langle    \by,{\bM^{-1}}  \by \rangle} \cdot
\]
\item If $\bY$ is an invariant subset of $\bH \bA^*$ and $\operatorname{ker}( {\bY}^*) \cap \operatorname{range}(\bH^{-1} \bZ)= \{ \mathbf{0}\}$, 
 then   
\[
\theta(\bA, \bH, \bH, \bY , \bZ) \geq  \frac{\lambda_{\min}(\bH \bM)}{\lambda_{\max}(\bH\bM)} \times  \operatorname{inf}\limits_{\by\in \ker ({\bY}^*)\setminus\{\mathbf 0\}} \frac{|\langle \bA^{-1} \by, \by \rangle|}{ \langle \bM^{-1} \by, \by \rangle} \cdot 
\]
\end{enumerate}
\end{theorem} 

Before we give the proof of the theorem, we observe that since the
matrices $\bH$ and $\bM$ are hpd, the eigenvalues of $\bH \bM$ are real and positive.  

\vspace*{2mm}
{\em Proof of Theorem~\ref{th:abscv2}.}
\begin{enumerate}
\item  
This case corresponds to Lemma~\ref{lem:Pdorth} where it is proved that only the condition  $\operatorname{ker}( \bY^*  ) \cap \operatorname{range}(\bH^{-1} \bZ)= \{ \mathbf{0}\}$ is necessary in order to ensure that WPD-GMRES does not break down. Here, that condition is equivalent to $\operatorname{ker}( \bZ^* \bA^* ) \cap \operatorname{range}( \bZ)= \{ \mathbf{0}\}$. To prove this, consider a vector $\bZ \bz$ in that intersection: $\bZ^* \bA^* \bZ \bz = 0$ implies that
\[
0 = \langle \bZ^* \bA^* \bZ \bz, \bz \rangle = \underbrace{\langle \bZ^* \bM \bZ \bz, \bz \rangle}_{\in \mathbb R} - \underbrace{\langle \bZ^* (\bA^* - \bM) \bZ \bz, \bz \rangle}_{\in \mathbb I} ,
\] 
where $\mathbb I$ stands for the imaginary axis.
Thus, $\langle \bZ^* \bM \bZ \bz, \bz \rangle=0$ and since $M$ is positive definite, $\bZ \bz=0$,
and this allows us to conclude that $\bz = \mathbf 0$ and that 
\mbox{$\operatorname{ker}( \bZ^* \bA^* ) \cap \operatorname{range}( \bZ)= \{ \mathbf{0}\}$}.

Letting $\by \in \range (\bP_D)$, i.e., $\by = \bP_D \by$, and using the fact that
$\bH \bP_D = \bP_D^* \bH$,
the numerator in \nref{theta:eq}
satisfies
\[
|\langle \bH \bP_D \bA  \bH \by, \by \rangle|^2 =   |\langle \bP_D^* \bH \bA  \bH \by, \by \rangle|^2 =  | \langle  \bH \bA  \bH \by, \bP_D\by \rangle|^2 =  | \langle  \bH \bA  \bH \by, \by \rangle|^2. 
\]
The first term in the denominator satisfies
\[
\| \bP_D  \bA\bH \by \|_{\bH}^2 \leq  \|   \bA\bH \by \|_{\bH}^2 ~\cdot
\]
Thus,
\begin{align*}
\theta(\bA, \bH, \bH, \bH \bA \bZ, \bZ) 
 &\geq  \operatorname{inf}\limits_{\by\in\range (\bP_D)\setminus\{\mathbf 0\}} \frac{|\langle  \bH \bA  \bH \by, \by \rangle|^2}{\| \bA  \bH \by\|_{\bH } \langle \bH \by, \by \rangle} \\ 
 &=  \operatorname{inf}\limits_{\by\in\range (\bH \bP_D)\setminus\{\mathbf 0\}} \frac{|\langle   \bA   \by, \by \rangle|^2}{ \|   \bA   \by \|_{\bH}^2  \| \by\|_{\bH^{-1}}^2}  \\ 
 &\geq  \operatorname{inf}\limits_{\by\in\range (\bA \bH \bP_D)\setminus\{\mathbf 0\}} \frac{|\langle   \bA^{-1}   \by, \by \rangle|}{ \langle    \by, \bH  \by \rangle} \times  \operatorname{inf}\limits_{\by\in\range (\bH \bP_D)\setminus\{\mathbf 0\}} \frac{|\langle   \bA   \by, \by \rangle|}{\langle \bH^{-1} \by, \by \rangle} \\   
 &\geq  \operatorname{inf}\limits_{\by\in\range (\bA \bH \bP_D)\setminus\{\mathbf 0\}} \frac{|\langle  \bA^{-1} \by, \by \rangle|}{ \langle    \by, \green{\bM^{-1}}  \by \rangle} \times \operatorname{inf}\limits_{\by\in\range (\bA \bH \bP_D)\setminus\{\mathbf 0\}} \frac{\langle   \green{\bM^{-1}}   \by, \by \rangle}{ \langle    \by, \bH  \by \rangle} \\ 
&\quad \quad  \times  \operatorname{inf}\limits_{\by\in\range (\bH \bP_D)\setminus\{\mathbf 0\}} \frac{\langle   \bM   \by, \by \rangle}{\langle \bH^{-1} \by, \by \rangle} \\   
 &\geq  \operatorname{inf}\limits_{\by\in\range (\bA \bH \bP_D)\setminus\{\mathbf 0\}} \frac{|\langle  \bA^{-1} \by, \by \rangle|}{ \langle    \by, \green{\bM^{-1}}  \by \rangle} \times \operatorname{inf}\limits_{\by\in\range \mathbb K^n \setminus\{\mathbf 0\}} \frac{\langle   \green{\bM^{-1}}   \by, \by \rangle}{ \langle    \by, \bH  \by \rangle} \\ 
&\quad \quad \times  \operatorname{inf}\limits_{\by\in \mathbb K ^n\setminus\{\mathbf 0\}} \frac{\langle   \bM   \by, \by \rangle}{\langle \bH^{-1} \by, \by \rangle} \cdot    
\end{align*}
In the fourth line it was used that  $\langle   \bM   \by, \by \rangle \leq | \langle   \bA   \by, \by \rangle |$.  The result in the theorem is proved by recognizing that the two last terms are Rayleigh quotients for the preconditioned operator $\bH \bM$, and recalling that $\range (\bP_D) = \ker(\bY^*) = \ker(\bZ^* \bA^* \bH) $ (Lemma~\ref{lem:PDQD}) so that $\range (\bA \bH \bP_D) =  \ker(\bZ^* \bA^* \bA^{-1})$.  
\item 
This case corresponds to Lemma~\ref{lem:invariant} where it is proved that GMRES does not break down as long as one of the conditions from Theorem~\ref{th:wellposed} is verified, \textit{e.g.}, $\operatorname{ker}( \bY^*  ) \cap \operatorname{range}(\bH^{-1} \bZ)= \{ \mathbf{0}\}$. Taking from the lemma that $\bQ_D \bH \bP_D = \bH \bP_D$,
using that $ \bP_D\bA = \bA \bQ_D $, and from the fact that $\by\in\range (\bQ_D)$ implies $\by = \bQ_D\by$,
we obtain 
\begin{align*}
\theta(\bA, \bH, \bH, \bY, \bZ) &=  \operatorname{inf}\limits_{\by\in\range (\bQ_D)\setminus\{\mathbf 0\}} \frac{|\langle \bA \by, \by \rangle|^2}{ \langle \bH \bA \by, \bA \by \rangle \langle \bH^{-1} \by, \by \rangle} \\
& \geq  \operatorname{inf}\limits_{\by\in\range (\bQ_D)\setminus\{\mathbf 0\}} \frac{|\langle \bA \by, \by \rangle|}{ \langle \bH \bA \by, \bA \by \rangle} \times \operatorname{inf}\limits_{\by\in\range (\bQ_D)\setminus\{\mathbf 0\}}  \frac{| \langle \bA \by, \by \rangle|}{\langle \bH^{-1} \by, \by \rangle} \\
&\geq  \operatorname{inf}\limits_{\by\in\range (\bA \bQ_D)\setminus\{\mathbf 0\}} \frac{|\langle \bA^{-1} \by, \by \rangle|}{ \langle \bH \by, \by \rangle} \times \operatorname{inf}\limits_{\by\in\range (\bQ_D)\setminus\{\mathbf 0\}} \frac{\langle \bM \by, \by \rangle}{\langle \bH^{-1} \by, \by \rangle} \\
&\geq  \operatorname{inf}\limits_{\by\in\range (\bA \bQ_D)\setminus\{\mathbf 0\}} \frac{|\langle \bA^{-1} \by, \by \rangle|}{ \langle \green{\bM^{-1}} \by, \by \rangle} \times \operatorname{inf}\limits_{\by\in\range (\bA \bQ_D)\setminus\{\mathbf 0\}} \frac{\langle \green{\bM^{-1}} \by, \by \rangle}{ \langle \bH \by, \by \rangle}\\ 
&\quad \quad \times \operatorname{inf}\limits_{\by\in\range (\bQ_D)\setminus\{\mathbf 0\}} \frac{\langle \bM \by, \by \rangle}{\langle \bH^{-1} \by, \by \rangle}\\ 
&\geq  \operatorname{inf}\limits_{\by\in\range (\bA \bQ_D)\setminus\{\mathbf 0\}} \frac{|\langle \bA^{-1} \by, \by \rangle|}{ \langle \green{\bM^{-1}} \by, \by \rangle} \times \operatorname{inf}\limits_{\by\in \mathbb K^n\setminus\{\mathbf 0\}} \frac{\langle \green{\bM^{-1}} \by, \by \rangle}{ \langle \bH \by, \by \rangle}\\ 
&\quad \quad \times \operatorname{inf}\limits_{\by\in \mathbb K^n \setminus\{\mathbf 0\}} \frac{\langle \bM \by, \by \rangle}{\langle \bH^{-1} \by, \by \rangle} \cdot 
\end{align*}
The result in the theorem is proved by recognizing that the two last terms are Rayleigh quotients for the preconditioned operator $\bH \bM$, and recalling that, from Lemma~\ref{lem:PDQD}, $\range (\bQ_D) = \ker (\bY^* \bA^{-1} )$ so $\range (\bA \bQ_D) = \ker (\bY^*)$. \cvd
\end{enumerate}

\begin{remark}
In the proof of Theorem~\ref{th:abscv}, in each case, the two matrices $\green{\bM^{-1}}$ can be replaced by any hpd matrix, say $\tilde{\bM}^{-1}$,
and the thesis of the theorem changed appropriately so that, e.g. for case 1,
\[
\theta(\bA, \bH, \bH, \bH \bA \bZ , \bZ) \geq  \frac{\lambda_{\min}(\bH \bM)}{\lambda_{\max}(\bH\tilde\bM)} \times  \operatorname{inf}\limits_{\by\in  \ker(\bZ^* \bA^* \bA^{-1} )\setminus\{\mathbf 0\}} \frac{|\langle  \bA^{-1} \by, \by \rangle|}{ \langle    \by,{\tilde{\bM}^{-1}}  \by \rangle} \cdot
\]
\end{remark}

In the case that $\bY = \bZ$ the two non-breakdown conditions are automatically verified if $\bA$ 
is positive definite and $\bH$ is hpd. For this reason the choice $\bY = \bZ$ is quite natural. 
However, there is no natural new lower bound of  the convergence factor $\theta$ in this special case.
What it does hold is that if $\bY = \bZ$ then we can choose $\bY$ as an invariant subset
of $\bH\bA$; see also Remark~\ref{case2:rem} below.

\section{A new spectral deflation space} \label{sec:Newspace}

In this section our objective is to compute $\bY$ and $\bZ$ in such a way that the convergence of WPD-GMRES is bounded explicitly. More precisely, following the results in Theorem~\ref{th:abscv} we aim to find a subset of vectors that satisfy 
\[
\frac{|\langle  \bA^{-1} \by, \by \rangle|}{ \langle    \by,{\bM^{-1}}  \by \rangle} \geq \gamma
\]
for some choice of $\gamma$. We will show that this can be done by computing eigenvectors of a particular generalized eigenvalue problem. 
Then we connect them to our WPD-GMRES bound by making explicit our choice of $\bY$ and $\bZ$.  

\subsection{Choice of deflation space and convergence of WPD-GMRES}

\begin{definition}[Deflation Space]
\label{def:Z}
Given a pd matrix $\bA$, let $(\lambda_{j}, \bz^{(j)})_{j = 1, \dots n}$ denote the eigenpairs of the generalized eigenvalue problem \eqref{eq:gevpNM}, \textit{i.e.}, $\bN \bz^{(j)} = \lambda_{j} \bM \bz^{(j)}$, with $\bM$ and $\bN$ the Hermitian and skew-Hermitian parts of $\bA$ as in \eqref{eq:splitA}. Let $\tau >0$. Consider the space 
\[
\mathcal Z = \spann  \{ \bz^{(k)} ; |\lambda_{k} | > \tau \}. 
\]
\end{definition}
From Lemma~\ref{lem:prelim} it follows that this generalized eigenvalue problem has a complete set of
eigenvactors, and this is used in the following result.
\begin{theorem}
\label{th:speceq}
Let $\mathcal Z$ be as in Definition~\ref{def:Z}. Its orthogonal complement is the space
$
\mathcal Z^\perp = \spann \{\bM \bz^{(k)} ; |\lambda_{k} | \leq \tau \}.
$
Moreover, any $\by \in \mathcal Z^\perp \setminus\{\mathbf{0}\}$ satisfies 
\[
 \frac{|\langle \bA^{-1} \by, \by \rangle|} { \langle \bM^{-1} \by, \by \rangle } \geq  \frac{1}{{ 1 + \tau^2}} \cdot
\]
\end{theorem}

\begin{proof}
Let $\by \in \mathbb K^n$. By Lemma~\ref{lem:prelim}, a set of $n$ eigenvectors $\bz^{(k)}$ can be chosen to form an $\bM$-orthonormal basis of $\mathbb C ^n$ so that 
\[
\by = \sum_{k=1}^n \beta_{k} \bM \bz^{(k)} \text{ with }  \beta_k   = \langle \by, \bz^{(k)} \rangle \in \mathbb C.
\]
The characterization of $\mathcal Z^\perp$ comes from
\[
 \by \in \mathcal Z^\perp \Leftrightarrow \left( \by \perp \bz^{(k)} \text{ if }  |\lambda_{k} | > \tau\right)   \Leftrightarrow \left(  \beta_k  = 0 \text{ if } |\lambda_{k} | > \tau\right) \Leftrightarrow \by \in \spann \{\bM \bz^{(k)} ; |\lambda_{k} | \leq \tau \}.
\]
Now, take any $\by = \sum_{k;  |\lambda_{k} | \leq \tau}  \beta_k  \bM \bz^{(k)} \in \mathcal Z^\perp $ (with $ \beta_k   = \langle \by, \bz^{(k)} \rangle$). By the factorization $ \bA = \bM (\matid +  \bM^{-1} \bN)$, we obtain 
\[
\bA^{-1} \by  = (\matid  + \bM^{-1} \bN)^{-1} \bM^{-1} \by = \sum_{k;  |\lambda_{k} | \leq \tau}  \beta_k  (\matid  + \bM^{-1} \bN)^{-1}  \bz^{(k)} = \sum_{k;  |\lambda_{k} | \leq \tau} \frac{ \beta_k }{ 1 + \lambda_{k}}  \bz^{(k)}
\] 
and
\[
\langle \bA^{-1} \by, \by  \rangle= \left\langle \sum_{k;  |\lambda_{k} | \leq \tau} \frac{ \beta_k }{ 1 + \lambda_{k}}  \bz^{(k)}, 
\sum_{k;  |\lambda_{k} | \leq \tau}  \beta_k  \bM \bz^{(k)} \right\rangle = \sum_{k;  |\lambda_{k} | \leq \tau} \frac{\overline  \beta_k   \beta_k }{ 1 + \lambda_{k}} \cdot 
\]
Thus, the term in the numerator is 
\[
| \langle \bA^{-1} \by, \by \rangle| = \left|  \sum_{k;  |\lambda_{k} | \leq \tau} 
\frac{\overline  \beta_k   \beta_k ( 1 - \lambda_{k})}{( 1 - \lambda_{k})( 1 + \lambda_{k})} \right| = \left|  \sum_{k;  |\lambda_{k} | \leq \tau} \frac{\overline  \beta_k   \beta_k ( 1 -  \lambda_{k})}{ 1 +  |\lambda_{k}|^2} \right|  \cdot
\]
Using that $\lambda_{k} = \pm i | \lambda_{k}|$ and that $| \langle \bA^{-1} \by, \by \rangle| \geq \max (|\Re( \langle \bA^{-1} \by, \by \rangle)|, |\Im( \langle \bA^{-1} \by, \by \rangle)|)$ we can bound this term as follows, 
\[
| \langle \bA^{-1} \by, \by \rangle| \geq 
 \max\left(
\sum_{k;  |\lambda_{k} | \leq \tau} \frac{\overline  \beta_k   \beta_k }{ 1 +  |\lambda_{k}|^2}  ,
\left| \sum_{k;  |\lambda_{k} | \leq \tau} \frac{   \lambda_{k}\overline  \beta_k   \beta_k }{ 1 +  |\lambda_{k}|^2}\right| 
\right)
\geq \frac{1}{1 + \tau^2}  \sum_{k;  |\lambda_{k} | \leq \tau} \overline  \beta_k   \beta_k . 
\]
On the other hand, the denominator can be rewritten as 
\[
\langle \bM^{-1} \by, \by \rangle  =   \langle \sum_{k;  |\lambda_{k} | \leq \tau}  \beta_k   \bz^{(k)} , \sum_{k;  |\lambda_{k} | \leq \tau} \bM  \beta_k   \bz^{(k)} \rangle  = \sum_{k;  |\lambda_{k} | \leq \tau} \overline  \beta_k   \beta_k . 
\]
We finally obtain the result by division and simplification by $\sum_{k;  |\lambda_{k} | \leq \tau} \overline  \beta_k   \beta_k $ (which is not $0$ unless $\by = \mathbf {0}$).  
\end{proof}

In the following we use the short-hand notation $\bA^{-*}$ for $(\bA^{*})^{-1} = (\bA^{-1})^*$.

\begin{theorem}
\label{th:abscv}
Let us assume that $\bA$ is positive definite, $\bH$ is hpd, and $\bW = \bH$. With $\mathcal Z$ as introduced in Definition~\ref{def:Z},
the quantity $\theta$ in the convergence result of WPD-GMRES  (Theorem~\ref{th:DRPconv})  can be bounded as follows. 
If either
\begin{enumerate}
\item  $\range(\bZ) = \mathcal Z$ and $\bY = \bH \bA \bZ$, \textit{i.e.}, $\bP_D$ is $\bH$-orthogonal,  or
\item $\range(\bY) = \mathcal Z$, $\bY$ is an invariant subset of $\bH \bA^*$, and  $\operatorname{ker}( \bY^*  ) \cap \operatorname{range}(\bH^{-1} \bZ)= \{ \mathbf{0}\}$.
\end{enumerate}
Then
\[
\theta(\bA, \bH, \bH,  \bY , \bZ) \geq  \frac{\lambda_{\min}(\bH \bM)}{\lambda_{\max}(\bH\bM)} \times \frac{1}{{1+\tau^2}} \cdot  
\]
\end{theorem} 

\noindent
{\em Proof.} 
In order to combine the results from Theorem~\ref{th:DRPconv} and Theorem~\ref{th:speceq}, it only remains to prove one identity in each case.  
\begin{enumerate}
\item 
$ 
\ker ( \bZ^* \bA^* \bA^{-1} ) =
\range(\bA^{-*} \bA \bZ)^\perp = \mathcal Z^\perp = \range( \bZ)^\perp$. 
Indeed, 
\[
\bA \bz^{(k)} = (1 + \lambda_k)\bM \bz^{(k)} \text{ and } \bA^* \bz^{(k)} = (1 - \lambda_k)\bM \bz^{(k)} 
\]
so $\frac{1}{1 + \lambda_k}   \bA \bz^{(k)} = \frac{1}{1 - \lambda_k}   \bA^* \bz^{(k)}$, \textit{i.e.}, 
\[
  \bA^{-*}  \bA \bz^{(k)} = \frac{1 + \lambda_k}{1 - \lambda_k}  \bz^{(k)} \text{ with }  1 + \lambda_k \neq 0 \text{ and } 1 - \lambda_k \neq 0.
\]
 
\item $\ker(\bY^*) = \range{\bY}^\perp = \mathcal Z^\perp$.  
~~~~~~~~~~~~~~  ~~~~~~~~~~~~~~~~ ~~~~~~~~~~~~~~~~~~~~~~~~~~~~~~~  ~~~~~~~~  ~~~~  \cvd
\end{enumerate}
 
\begin{remark} \label{case2:rem}
It must be noted that, in the case labeled 2, the two first assumptions are in general not compatible: $ \mathcal Z$ is not necessarily an invariant subset of $\bH \bA^*$. An exception is the case where $\bH = \bM^{-1}$, \textit{i.e.}, the problem is preconditioned by the inverse
of the Hermitian part of $\bA$. Then $\bH \bA^* \mathcal Z = \mathcal Z$ since, for any eigenvector $\bz^{(k)}$, we have
\[
\bH \bA^* \bz^{(k)} = \bH (1 - \lambda^{(k)}) \bM \bz^{(k)} =  \bM^{-1} (1 - \lambda^{(k)}) \bM \bz^{(k)} =  (1 - \lambda^{(k)}) \bz^{(k)}.    
\] 
The convergence result then holds for any $\bZ$ such that  $\operatorname{ker}( \bY^*  ) \cap \operatorname{range}(\bH^{-1} \bZ)= \{ \mathbf{0}\}$ , \textit{e.g.}, $\bZ = \bY$ or $\bZ = \bM \bY$ or $\bZ = \bM^{-1} \bY$.  
\end{remark}

\begin{remark}
The case where no vectors are deflated corresponds to setting $\tau = \rho(\bM^{-1} \bN)$ (the spectral radius). In that case, the estimate in Theorem~\ref{th:abscv} is exactly the estimate in \cite[Corollary 4.4]{spillane2023hermitian}. 
\end{remark}

\subsection{Real-valued  Case}
\label{sec:real}

We consider now the case where 
$\bA$ and $\rhs$ are real-valued. The solution will also be real-valued and the iterative solver should be applied in $\mathbb R$. In this case, the next theorem proposes an alternate basis for the deflation space from Definition~\ref{def:Z}, for which the deflation operators $\bP_D$ and $\bQ_D$ are real. 

\begin{theorem}[Deflation Space (Real-valued case)]
\label{th:realZ}
Given a pd real matrix $\bA$, let $(\lambda_{j}, \bz^{(j)})_{j = 1, \dots n}$ denote the eigenpairs of generalized eigenvalue problem \eqref{eq:gevpNM}, \textit{i.e.}, $\bN \bz^{(j)} = \lambda_{j} \bM \bz^{(j)}$, with $\bM$ and $\bN$ the Hermitian and skew-Hermitian parts of $\bA$ as in \eqref{eq:splitA}. Let $\tau >0$. The deflation space $\mathcal Z$ from Definition~\ref{def:Z} can also be written as 
\[
\mathcal Z = \spann  \{ \{\Re(\bz^{(k)}), \Im(\bz^{(k)})\} ; |\lambda_{k} | > \tau  \}. 
\]
Since $\lambda_k = \pm i \mu$, there are two eigenvectors with $|\lambda_k| = \mu$. Therefore, it suffices to choose the real and imaginary part of one of them to span the same subspace.
\end{theorem}
\begin{proof}
If $\bA$ is real, the non-zero eigenvalues come in complex conjugate pairs. Indeed, let $(\lambda, \bz)$  denote an eigenpair of the generalized eigenvalue problem \eqref{eq:gevpNM} with $\lambda \neq 0$. Then $\lambda = i \mu$ where $\mu \in \mathbb R$, and it follows by taking the complex conjugate of \eqref{eq:gevpNM} that 
\[
\overline{\bN \bz} = \overline{i \mu \bM \bz} \Leftrightarrow  \bN \overline \bz =- i \mu \bM \overline \bz. 
\]
So the complex-conjugate $\overline \bz$ is an eigenvector corresponding to eigenvalue $- i \mu = -\lambda$. We conclude by noticing that the space spanned by $\bz$ and $\overline{\bz}$ is the same as the space spanned by $\Re(\bz)$ and $\Im(\bz)$. 
\end{proof}

In other words, we choose as our deflation space, the real vectors which are the real and imaginary parts of the eigenvectors of the generalized
problem \eqref{eq:gevpNM} corresponding to eigenvalues larger than $\tau$ in modulus.

\section{Numerical experiments}
\label{sec:Numerical}
We present three different types of numerical experiments illustrating how the proposed deflation
space, inspired by our new bounds, can reduce the number of iterations needed for convergence.
In order for the proposed deflation to be practical, either  the cost of computing the
eigenvectors of the generalized skew-Hermitian problem (\ref{eq:gevpNM}) is offset by the
saving in number of iterations, or, as in our second example, when GMRES fails to converge,
but with a relatively small deflation space, convergence is attained.

Since our problems are real-valued ($\mathbb K = \mathbb R$), Hermitian means symmetric. Moreover, the eigenvectors can be grouped into complex conjugate pairs and we apply the strategy from Theorem~\ref{th:realZ} to obtain real-valued deflation operators. The Octave command is 
\texttt{ $\mathbf{Z}$ = [real($\mathbf{V}$(:,1:2:m)), imag($\mathbf{V}$(:,1:2:m))]}, where the columns of $\mathbf{V}$ are the sorted eigenvectors and $\text{m}$ is the (even) number of vectors that should form the deflation space. Finally, we set $\bY = \bH \bA \bZ$ and compute the deflation operators as in Definition~\ref{def:PDQD}. 

\paragraph{Quantities of interest.}

We report the number of iterations needed for WPD-GMRES to achieve convergence from a zero initial vector, with the preconditioners and deflation operators introduced above, and the weight matrix $\bW = \bH$. The stopping criterion is 
$\|\br_i\|_\bH < 10^{-10} \|\bb \|_\bH$.
We also report the upper bound for $\theta$  predicted by Theorem~\ref{th:abscv} (case 1) as well as the corresponding experimental value
\[
\theta_{th} =  \frac{1}{\kappa(\bH\bM)} \times \frac{1}{{1+|\lambda_{m+1}|^2}} \text{ and } \theta_{exp} = \min\limits_i \left\{ 1 - \frac{\| \br_{i+1} \|_\bH^2}{\| \br_{i} \|_\bH^2}\right\}.
\]  
The threshold $\tau $ in the theoretical bound has been substituted for $|\lambda_{m+1}|$, the modulus
of the largest eigenvalue not used for the deflation space. In order to compute $\theta_{th}$, $\kappa(\bH\bM)$ is approximated by solving a linear system for matrix $\bM$ preconditioned by $\bH$ with PCG and taking the ratio of the Ritz values. The theorem guarantees that $\theta_{th} \leq  \theta_{exp}$ and this has indeed been the case for every single one of our numerical experiments.

\subsection{Numerical Illustration: Jordan block}
\label{sec:Jordan}
We first consider an academic problem. Namely, a $1000 \times 1000$ (bidiagonal) scaled Jordan block of the form 
\[
\bA = \begin{pmatrix}
    1 & \alpha & & & & \\
    & 1 & \alpha &&& \\
    &&\ddots & \ddots &&\\
    &&& 1 & \alpha \\
    & & & & 1 
\end{pmatrix}
\]
for which the symmetric part $\bM$ is a tridiagonal Toeplitz matrix. Choosing $\alpha = 0.99$ guarantees that all eigenvalues of the symmetric part $\bM$ are positive so $\bA$ is positive definite; see, e.g.,
\cite[1.4.I16]{hornjoh:85}. The condition number of the symmetric part is $\kappa (\bM) = 199$. 
Our objective is to illustrate one of the main contribution of the article which is the new deflation space. For this reason we do not consider preconditioning or weighting, {i.e.}, $\bW = \bH = \matid$. Figure~\ref{fig:Jordan} shows the convergence curves obtained without deflation and with deflation of a varying number $m$ of vectors. This confirms that the vectors arising from the generalized eigenvalue problem \eqref{eq:gevpNM} are good candidates for deflation. Deflating 50 eigenvectors decreases the iteration count by 348 iteration and deflating 100 eigenvectors decreases the iteration count by 600. 
In Table~\ref{tab:Jordan} we show for various values of $m$, the dimension of the deflation space,
number of iterations to converge, 
the quantities
$\theta_{exp} $
calculated from the experiment,  and $\theta_{th}$ the bound for them as in
Theorem~\ref{th:abscv2}.

\begin{figure}
\begin{center}
\includegraphics[width=0.89\textwidth]{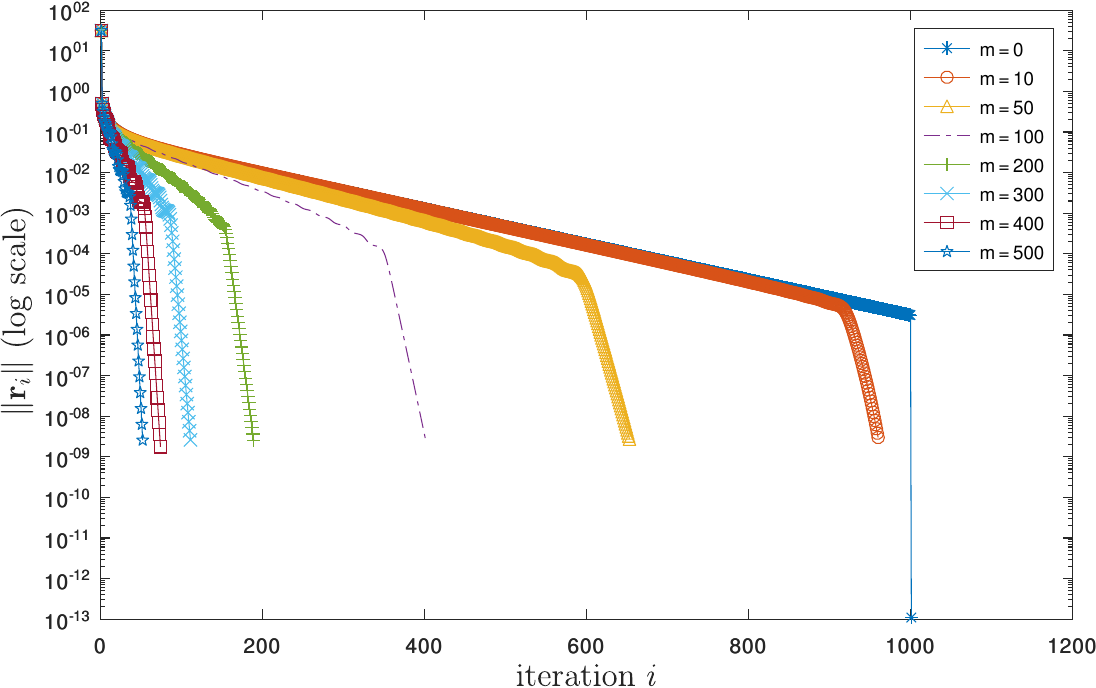}
\end{center}
\caption{Convergence for the scaled Jordan block example (Section~\ref{sec:Jordan}). The number $m$ of deflated vectors varies. Increasing $m$ accelerates convergence as predicted by the theory.}
\label{fig:Jordan}
\end{figure}

\begin{table}[]
    \centering
    \begin{tabular}{c|ccc}
$m$& iter& $\theta_{th}$& $\theta_{exp}$ \\
0 & 1000 & 1.00e-04 & 1.99e-02 \\
10 & 959 & 1.02e-04 & 1.99e-02 \\
50 & 652 & 1.33e-04 & 1.99e-02 \\
100 & 400 & 2.25e-04 & 1.99e-02 \\
200 & 188 & 5.79e-04 & 2.00e-02 \\
300 & 110 & 1.13e-03 & 1.99e-02 \\
400 & 73 & 1.81e-03 & 2.09e-02 \\
500 & 51 & 2.58e-03 & 2.38e-02 \\
    \end{tabular}
    \caption{Scaled Jordan block example (Section~\ref{sec:Jordan}). $m$ is the number of deflated vectors, iter is the iteration count, $\theta_{th}$ and $\theta_{exp}$ are the theoretical lower bound and observed worse value of $\theta$.}                     
    \label{tab:Jordan}
\end{table}

\subsection{Numerical Illustration: Stokes problem}
\label{sec:Stokes}
We consider the incompressible Stokes equation as in \cite[Section 6.2]{Guducu.Liesen.Mehrmann.Szyld.22}. The problem is discretized using the $\mathbb{Q}_1-\mathbb{Q}_1$ finite element approximation of the (unsteady) channel domain problem in IFISS \cite{ifiss}. The system matrices are of the form $\bA = \bM + \bN$ with
\[
\bM = \begin{pmatrix}
\mathcal M - \tau/2 \bA_H & \mathbf{0} \\
\mathbf{0} & - \tau/2 \mathbf C
\end{pmatrix}
\text{ and }
\bN = \begin{pmatrix}
\mathbf 0 & - \tau/2 \mathbf B \\
\tau/2 \mathbf B^* & 0
\end{pmatrix}
\]
where $\tau/2= 10^{-4}$ is the time step, $\mathcal M$ is the mass matrix, $\bA_H$ is the discretization of the second order term in the PDE, $\mathbf C$ is a stabilization term and $\mathbf B$ is the discrete divergence. The total number of degrees of freedom is $n = 12,675$ and the size of the $\mathbf B$ block is $8,450 \times 4,225$. The convergence curves are presented in Figure~\ref{fig:Stokes} where it can be observed that deflation of 1000 vectors allows to converge to a relative tolerance of $10^{-10}$ in 877 iterations whereas the non deflated problem does not converge. Table~\ref{tab:Stokes} shows that the theoretical $\theta_{th}$ does not vary much with the size $m$ of the deflation space. The condition number of the Hermitian part is very large: $\kappa(\bM) = 8.93e6$.

\begin{table}[]
    \centering
    \begin{tabular}{c|ccc}
$m$& iter& $\theta_{th}$& $\theta_{exp}$ \\
0 & 1000 & 1.12e-07 & 2.07e-04 \\
10 & 1000 & 1.12e-07 & 3.41e-03 \\
50 & 1000 & 1.12e-07 & 8.43e-03 \\
500 & 1000 & 1.12e-07 & 4.08e-03 \\
1000 & 877 & 1.12e-07 & 2.28e-03 \\
    \end{tabular}
    \caption{Stokes example (Section~\ref{sec:Stokes}). $m$ is the number of deflated vectors, iter is the iteration count, $\theta_{th}$ and $\theta_{exp}$ are the theoretical lower bound and observed worse value of $\theta$.}                     
    \label{tab:Stokes}
\end{table}

\begin{figure}
\begin{center}
\includegraphics[width=0.89\textwidth]{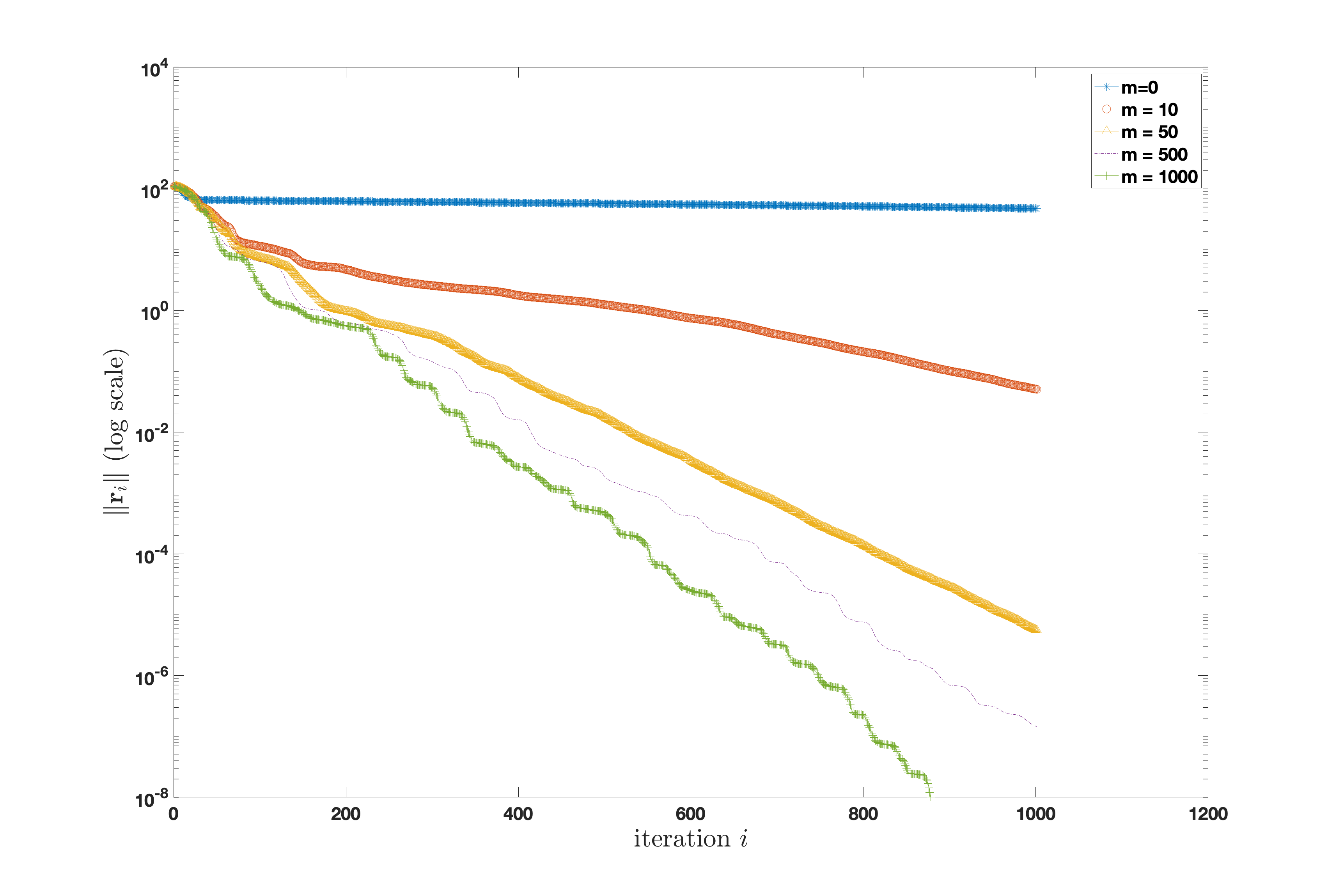}
\end{center}
\caption{Convergence for the Stokes problem (Section~\ref{sec:Stokes}). The number $m$ of deflated vectors varies. Increasing $m$ improves convergence as predicted by the theory.}
\label{fig:Stokes}
\end{figure}

\subsection{Numerical Illustration: Convection-Diffusion-Reaction}
\label{sec:CDR}
In this section, the problem considered is the convection-diffusion-reaction problem posed in $\Omega = [-1,1]^2$.  The strong formulation of the problem is: 
\begin{align*}
c_0 u + \operatorname{div}(\mathbf a u) - \operatorname{div} (\nu \nabla u) &= f \text{ in } \Omega,\\ 
u &= 0 \text{ on } \partial \Omega.
\end{align*}

The variational formulation is: 
Find $u \in H^1_0(\Omega)$ such that 
\[
\underbrace{\int_\Omega \left(\left(c_0 + \frac{1}{2} \operatorname{div} \mathbf a \right) uv+ \nu \nabla u \cdot \nabla v \right)}_{\text{symmetric part}}  + \underbrace{\int_\Omega \left(\frac{1}{2} \mathbf a \cdot \nabla u v - \frac{1}{2} \mathbf a \cdot \nabla v u \right)}_{\text{skew-symmetric part}}  = \int_\Omega fv, 
\]
for all $ v \in H^1_0(\Omega)$. The reaction coefficient $c_0 >0$ and viscosity $\nu >0 $ are assumed to be constant over $\Omega$. 

The right hand side is chosen as 
\[
f(x,y) =  \operatorname{exp} (-2.5(x ^2 + (y +0.8)^2)). 
\]
The convection field is parametrized by a constant $\eta \in \mathbb R$ and takes the values  
\[
\mathbf a(x,y) = \eta \pi \begin{pmatrix} - y - 0.8 \\  x  \end{pmatrix}. 
\]
It can be remarked that $ \operatorname{div} \mathbf a =0$. This is the same PDE as in \cite{spillane2023hermitian} with a change of variables.

\begin{figure}
\begin{center}
\includegraphics[width=0.29\textwidth]{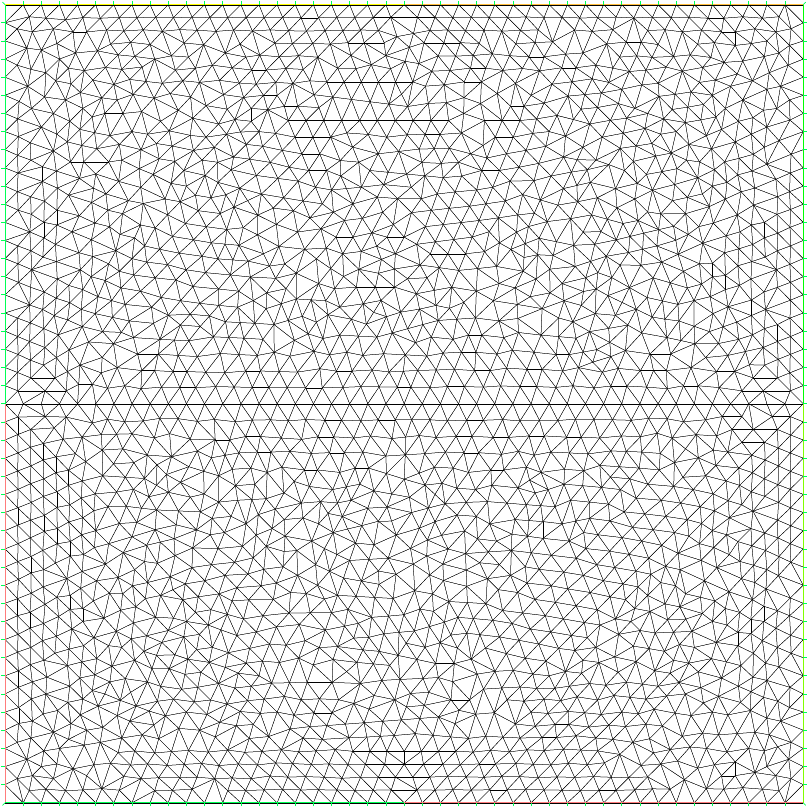}
\includegraphics[width=0.29\textwidth]{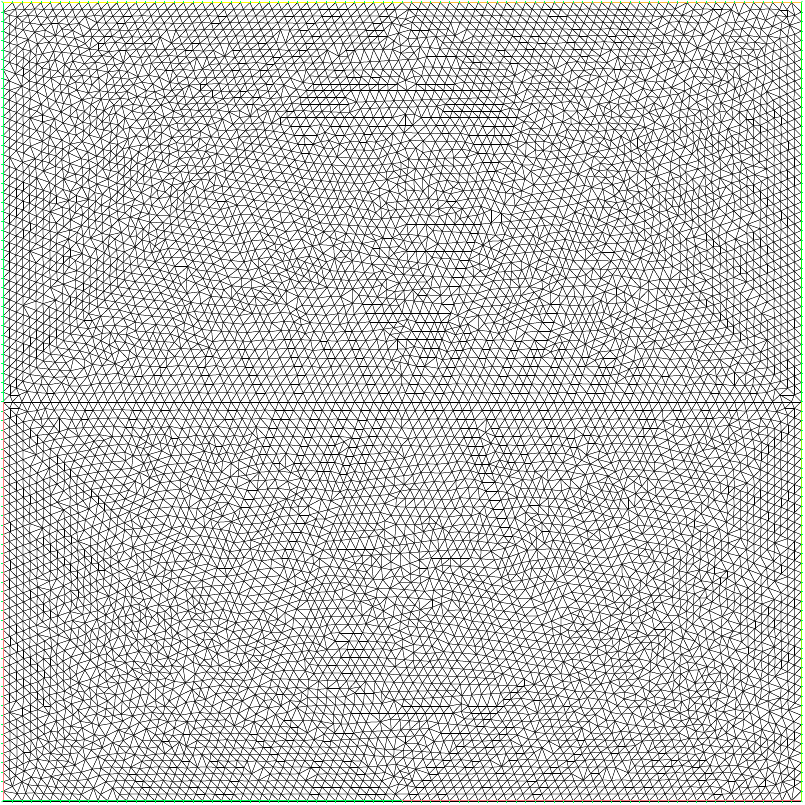}
\end{center}
\caption{Example meshes for the convection-diffusion-reaction problem (Section~\ref{sec:CDR}). Left: 2373 vertices and 4568 triangles. Right: 8643 vertices and 16948 triangles.}
\label{fig:meshnloc11}
\end{figure}

\begin{figure}
\begin{center}
\includegraphics[width = 0.33 \textwidth]{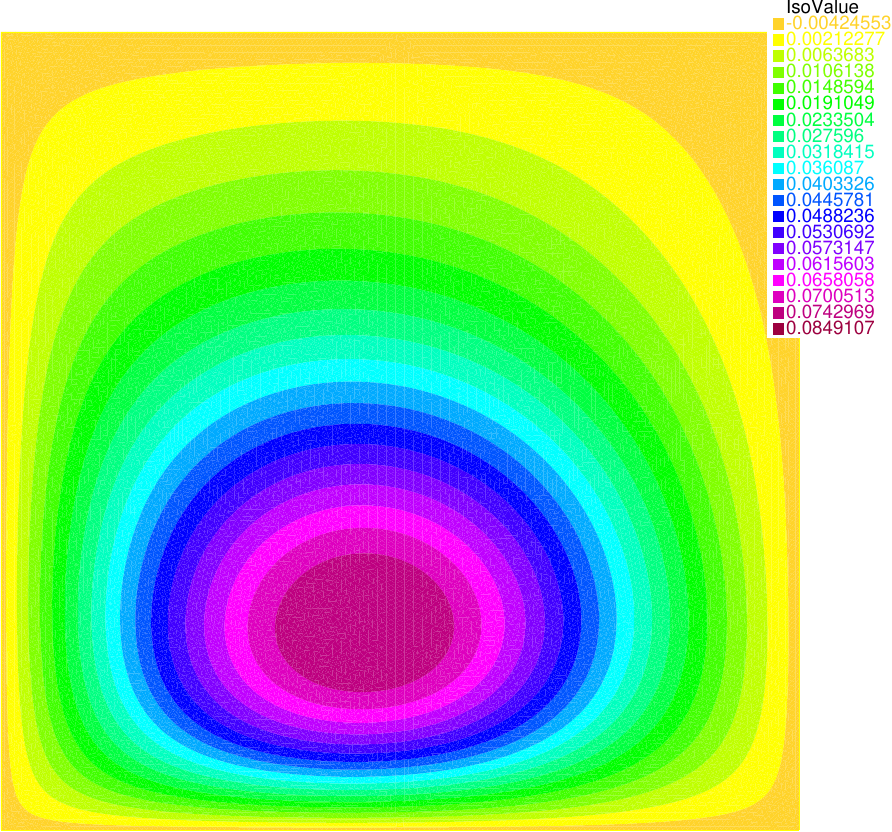}
\includegraphics[width = 0.29 \textwidth]{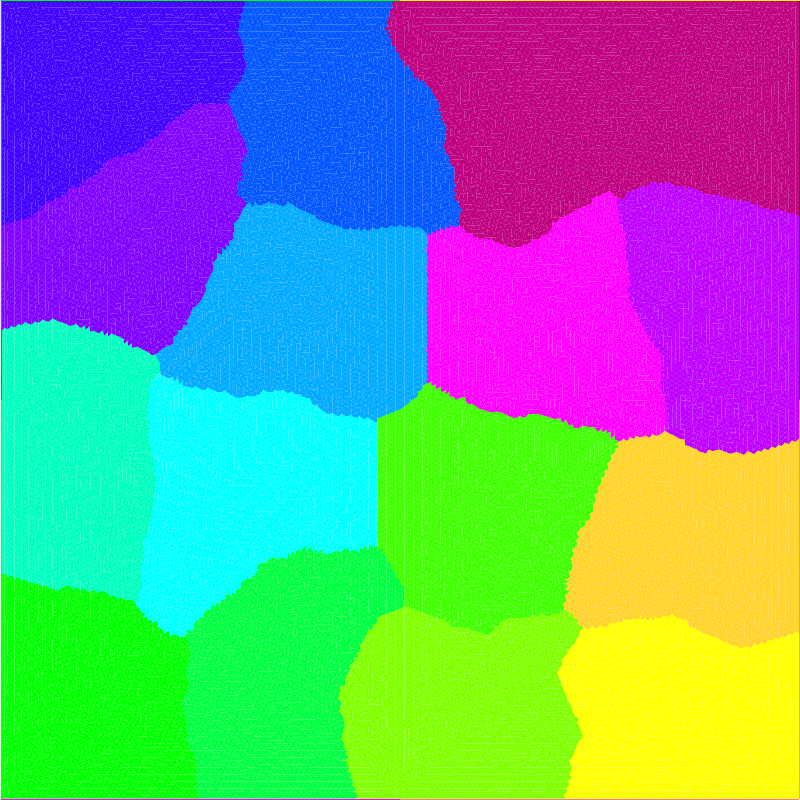}
\end{center}
\caption{Convection-diffusion-reaction problem (Section~\ref{sec:CDR}). Left: Solution. Right: Partition into 16 subdomains for $\HDD$.}
\label{fig:sol}
\end{figure}

The problem is discretized by Lagrange $\mathbb P_1$ finite elements on a triangular mesh. Two example meshes are shown in Figure~\ref{fig:meshnloc11} with different levels of refinement. They are good representatives of the meshes used throughout our numerical testing. We have deliberately not chosen a regular mesh since this assumption is not required by our theory. The solution is shown in Figure~\ref{fig:sol} (left). 
The WPD-GMRES algorithm is implemented in Octave while the finite element matrices are assembled by FreeFem++ \cite{MR3043640}.  All iteration counts for WPD-GMRES correspond to the number of iterations needed to reach $\|\br_i\|_\bH < 10^{-10} \|\bb \|_\bH$ starting from a zero initial vector. 
The Dirichlet boundary condition has been enforced by elimination. 
Let $(\phi_i)_{1 \leq i \leq n}$ denote the $\mathbb P_1$ finite element basis corresponding to the mesh. The problem matrix splits into
\[
\bA = \bM + \eta \tilde\bN, \text{ with } \bM \text{ spd and } \tilde\bN \text{ skew-symmetric},
\]
where the entries of $\bM$ and $\tilde\bN$ are 
\[
\bM_{ij} = \int_\Omega \left(c_0 \phi_i \phi_j + \nu \nabla \phi_i \cdot \nabla \phi_j \right),
\] 
and 
\[
\tilde\bN_{ij}=\int_\Omega \left(\frac{1}{2} \mathbf {\underline a} \cdot \nabla \phi_i \phi_j - \frac{1}{2} \mathbf{\underline a} \cdot \nabla \phi_j \phi_i \right), \text{ with }  \mathbf{\underline a}(x,y) = 2 \pi \begin{pmatrix} 0.1 - y \\  x-0.5 \end{pmatrix}. 
\]
The positive definiteness of $\bM$ is guaranteed by the assumption that $c_0$ and $\nu$ are positive.

\paragraph{Choice of preconditioners}

For our numerical study, three preconditioners are considered:
\begin{itemize}
\item $\bH = \matid$ the identity matrix,
\item $\bH = \Minv$ the inverse of $\bM$, the symmetric part of $\bA$, and
\item $\bH = \HDD$, a domain decomposition (DD) preconditioner based on a partition of the mesh into $N = 16$ subdomains (as shown in Figure~\ref{fig:sol}--Right). 
\end{itemize}

The choice of $\bH = \Minv$ was used, e.g., in \cite{zbMATH05074484}. It is also a fundamental feature of
the CGW method \cite{ConGol76,SzyWid93, zbMATH03619254}, and was successfully used recently for the
solution of Port-Hamiltonian systems~\cite{Guducu.Liesen.Mehrmann.Szyld.22}. In our experiments we used a direct solver (Octave's backslash) to apply $\bM^{-1}$.

For $\bH = \HDD$, The condition number of the resulting preconditioned operator is bounded by
\[
\kappa(\HDD \bM) \leq k_0 \left(1 + \frac{k_0}{\tau'}\right), 
\] 
where $k_0$ denotes the maximal number of subdomains that each mesh element belongs to \cite[Theorem 4.40]{spill2014} 
and  $\tau'$ is a parameter that has been set to $0.15$. The constant in the bound does not depend on the total number $N$ of subdomains or the mesh parameter $h$. 

In detail, $\HDD$ is the Additive Schwarz domain decomposition method with the GenEO coarse space \cite{2011SpillaneCR,spillane2013abstract}. The partition of $\Omega$ into $N$ subdomains $\Omega_s$ is computed automatically by Metis. One layer of overlap is added to each $\Omega_s$. 
Letting ${\bR_s}^\top$ ($s=1,\dots,N$) denote the prolongation by zero of local finite element functions (in $\Omega_s$) to the whole of $\Omega$, the preconditioner can be written as
\[
\HDD = \bPi \sum\limits_{s=1}^N \bR_s^\top \underbrace{(\bR_s \bM \bR_s^\top)^{-1}}_{\text{local solves}} \bR_s \bPi^\top + \bR_0^\top \underbrace{(\bR_0 \bM \bR_0^\top)^{-1}}_{\text{coarse solve}} \bR_0, 
\]
where $\bPi = \matid -  \bR_0^\top (\bR_0 \bM \bR_0^\top)^{-1} \bR_0 \bM$ is the coarse projector (also known as a deflation operator) and the vectors in $\bR_0^\top$ span the coarse space (or deflation space). The particularity of GenEO is that the coarse vectors are constructed by solving the low frequency eigenmodes for a generalized eigenvalue problem in each subdomain. 

\paragraph{Deflation Operator}

We aim to illustrate the convergence result in Theorem~\ref{th:abscv}. Once $\bM$ and $\bN$ have been assembled by FreeFem++, they are imported into Octave. The generalized eigenvalue problem~\eqref{eq:gevpNM} (\textit{i.e.}, $\bN \bz^{(j)} = \lambda_{j} \bM \bz^{(j)}$) is partially solved by \textit{eigs}: the eigenpairs corresponding to the eigenvalues of largest magnitude are approximated. 

\begin{remark}
When $\bH = \HDD$ is applied, the preconditioner already contains a deflation operator $\bPi$ which corresponds to a deflation space formed by eigenvectors of frequency less than a chosen $\tau'$ of well chosen eigenproblems in the subdomains. The presence of a deflation operator and deflation space in $\HDD$ does not interfere with the computation of a deflation space and deflation operator following the definition in Definition~\ref{def:Z}.  
\end{remark}

\begin{figure}[hbt]
\centering
\includegraphics[width=0.7 \textwidth]{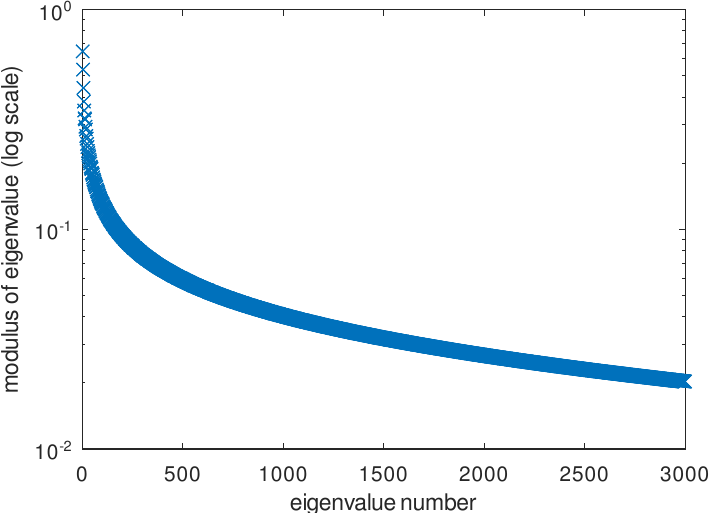}
\caption{Convection-diffusion-reaction problem (Section~\ref{sec:CDR}). $|\lambda_{1}|$ to $|\lambda_{3000}|$ in log scale (solution of~\eqref{eq:gevpNM} for $\eta = 1$).}
\label{fig:spectrumNM} 
\end{figure}

\begin{table}[hbt]
\begin{tabular}{c|lllllll}
m & 0 & 10 & 50 & 100 & 200 & 500 & 1000 \\
\hline
$|\lambda_{\text{m}+1}|$ &    0.65 &  0.32 &    0.18 &  0.13 &  0.09 &  0.06 &  0.04 \\
$\theta_{th}$ for $\bH = \Minv$ & 0.71 &  0.91 &    0.97 &   0.98 &  0.991 &  0.997 &  0.998 \\
$\theta_{th}$ for $\bH = \HDD$  &  0.043 &  0.056 &   0.0597 &  0.0605 &  0.0610 &  0.0614 &  0.0615 \\
$\theta_{th}$ for $\bH = \matid$   &   7.1e-05 &  9.2e-05&     9.8e-05  & 9.9e-05 &  1.001e-04 &  1.007e-04 &  1.009e-04 \\ 
\end{tabular}
\caption{Convection-diffusion-reaction problem (Section~\ref{sec:CDR}). If m is the rank of the deflation space, the second line is the modulus of the first eigenvalue not included in the deflation space, and the next three lines correspond to the theoretical bound $\theta_{th}$ for out three choices of preconditioner when $\eta =1$.}
\label{tab:lambdas}
\end{table}

\paragraph{Results for $\eta = 1$}

For this test case, we use a mesh with 63658 triangles and 32158 vertices. Once the homogeneous Dirichlet boundary condition has been treated by elimination, the problem has 31502 degrees of freedom. The parameter in the convection field is $\eta= 1$ so that $\bA = \bM + \bN$, 
with $\bN = \tilde \bN$. We first solve generalized eigenproblem~\eqref{eq:gevpNM}. The spectrum is represented in Figure~\ref{fig:spectrumNM} where the magnitudes of the 3000 largest (purely imaginary) eigenvalues are represented. Unfortunately,  the distribution of eigenvalues does not consist in a cluster of large eigenvalues and a tail of eigenvalues tightly clustered around 0. This would indeed have been ideal for selecting a value of $\tau$ to compute the deflation space with the formula in Definition~\ref{def:Z}. Table~\ref{tab:lambdas} gives the value of $|\lambda_{\text{m}+1}|$ for a few choices of $\text{m}$ as well as the resulting bounds for $\theta_{th}$. In these bounds we have injected the approximate condition number $\kappa(\bH\bM)$ which takes the following values
\[
\kappa(\matid \bM) = 9896.4 , \quad \kappa(\HDD \bM) = 16.241, \quad \text{and }  \kappa(\Minv \bM) = 1.  
\]

We now solve the problem with all three preconditioners and varying ranks $\text{m}$ of the deflation spaces. The case $\text{m} = 0$ corresponds to weighted and preconditioned GMRES with no deflation. The largest deflation space has rank $1000$ which corresponds to 3.2\% of the total number of 
degrees of freedom (dofs). 
The convergence curves are shown in Figure~\ref{fig:eta1}. We observe that our choice of deflation space indeed improves convergence: when the deflation space gets larger, the number of iterations is reduced. Since the case is only 
mildly nonsymmetric ($\rho(\bM^{-1} \bN) = 0.65$), the two \textit{good} preconditioners for $\bM$ also give good results for the full problem even with $\text{m} = 0$. For $\bH = \matid$, though, the effect of deflation is welcome: deflating $10$ vectors reduces the iteration count from 671 to 543 and deflating 50 vectors reduces the iteration count from 671 to 440.

\begin{figure}
\begin{minipage}{0.66\textwidth}
\includegraphics[width=\textwidth]{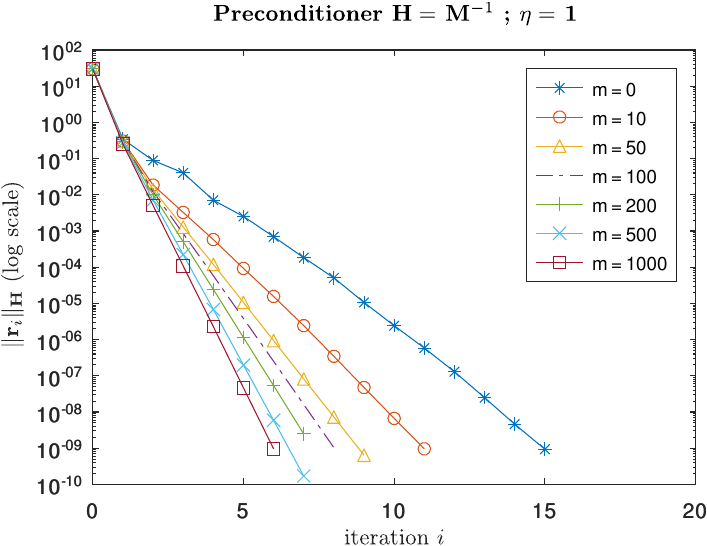} 
\end{minipage}
\begin{minipage}{0.3\textwidth}
\begin{tabular}{rrcc}
m & iter &  $\theta_{th}$ & $\theta_{exp}$  \\ 
0 & 15 & 7.1e-01 & 8.0e-01 \\ 
10 & 11 & 9.1e-01 & 9.7e-01 \\ 
50 & 9 & 9.7e-01 & 9.9e-01 \\ 
100 & 8 & 9.8e-01 & 9.95 e-01\\ 
200 & 7 & 9.91e-01 &9.97 e-01 \\ 
500 & 7 & 9.96e-01 & 9.99e-01 \\ 
1000 & 6 & 9.98e-01 & 9.99e-01 \\ 
\end{tabular}
\end{minipage}
   \\
\begin{minipage}{0.66\textwidth}
   \includegraphics[width=\textwidth]{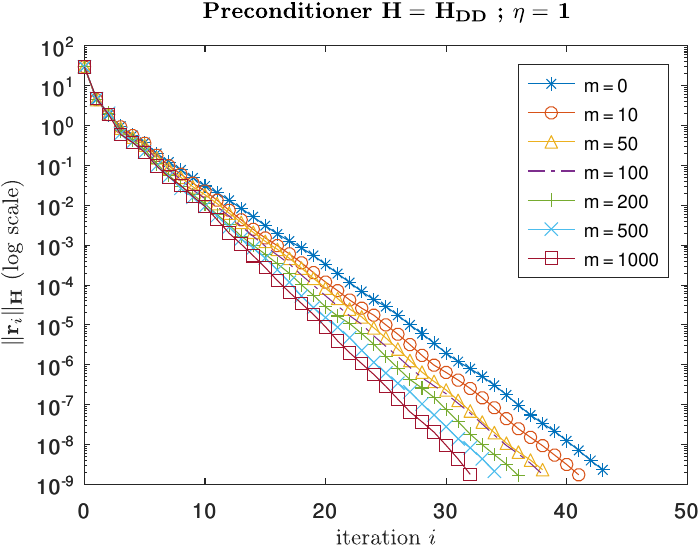}
\end{minipage}
\begin{minipage}{0.30\textwidth}
\begin{tabular}{rrcc}
m & iter &  $\theta_{th}$ & $\theta_{exp}$  \\ 
0 & 43 & 4.3e-02 & 5.3e-01 \\
10 & 41 & 5.6e-02 & 5.5e-01 \\ 
50 & 38 & 6.0e-02 & 6.0e-01 \\ 
100 & 38 & 6.1e-02 & 6.0e-01 \\ 
200 & 36 & 6.1e-02 & 5.8e-01 \\ 
500 & 34 & 6.1e-02 & 6.2e-01 \\ 
1000 & 32 & 6.1e-02 & 6.0e-01 \\ 
\end{tabular}
\end{minipage}
 \\
\begin{minipage}{0.66\textwidth}
 \includegraphics[width=\textwidth]{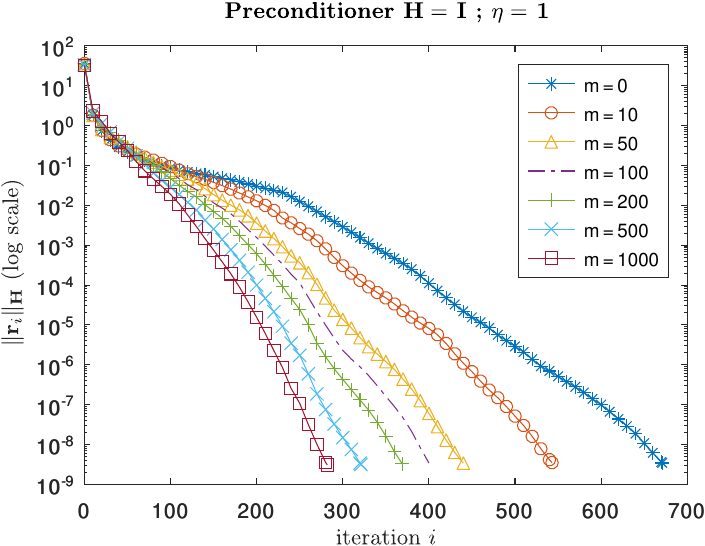}
\end{minipage}
\begin{minipage}{0.3\textwidth}
\begin{tabular}{rrcc}
m & iter &  $\theta_{th}$ & $\theta_{exp}$  \\ 
0 & 671 & 7.1e-05 & 1.7e-02 \\ 
10 & 543 & 9.2e-05 & 3.0e-02 \\ 
50 & 440 & 9.8e-05 & 4.2e-02 \\ 
100 & 400 & 9.9e-05 & 4.8e-02 \\ 
200 & 369 & 1.0e-04 & 5.2e-02 \\
500 & 321 & 1.0e-04 & 6.8e-02 \\ 
1000 & 282 & 1.0e-04 & 8.0e-02 \\ 
\end{tabular}
\end{minipage}
 \caption{Convergence curves when $\eta = 1$ in Section~\ref{sec:CDR} with varying preconditioner and rank of deflation space.}
\label{fig:eta1}
\end{figure}

\paragraph{Results for $\eta = 100$}

Next we change the value of $\eta$ to $100$. The global matrix is now $\bA = \bM + \bN$ with $\bN = 100 \tilde \bN$,
and the problem is much more nonsymmetric. The eigenvalues arising from~\eqref{eq:gevpNM} get multiplied by $100$ which has the effect of seriously deteriorating the bounds for $\theta_{th}$. Figure~\ref{fig:eta100} shows the convergence curves in this case. 
Again, we confirm that applying a preconditioner tailored for the symmetric part is a good idea 
(with 341 or 352 iterations instead of 1276). 
In the non-preconditioned case, it occurs that the non-deflated problem converges the fastest. This is surprising but does not contradict the theory. Figure~\ref{fig:zoomprecIeta100} shows the first 1000 residuals in this case for m$=0$ and m$=100$. It can be observed that deflation initially accelerates convergence, particularly where it is slowest. However, after roughly 500 iterations, the non-deflated algorithm is faster. For $\bH = \HDD$ or $\Minv$, convergence improves when more vectors are added to the deflation space. With $\bH = \Minv$, deflating 10 vectors reduces the iteration count from 
341 to 331 and deflating 100 vectors reduces it further to 259. With $\bH = \HDD $, deflating 10 vectors reduces the iteration count from 352 to 343 and deflating 100 vectors reduces it further to 275.

\begin{figure}
\begin{minipage}{0.66\textwidth}
\includegraphics[width=\textwidth]{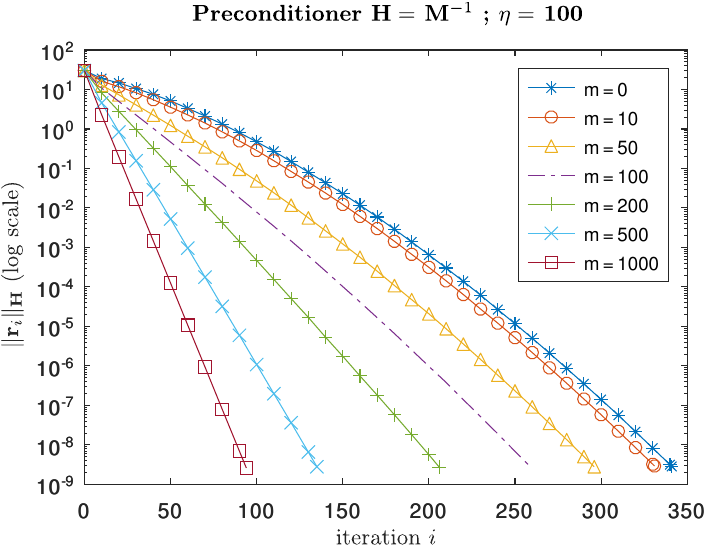}
\end{minipage}
\begin{minipage}{0.3\textwidth}
\begin{tabular}{rrcc}
m & iter &  $\theta_{th}$ & $\theta_{exp}$  \\ 
0 & 341 & 2.4e-04 & 3.1-03 \\ 
10 & 331 & 9.8e-04 & 4.5e-02 \\ 
50 & 296 & 3.2e-03 & 8.4e-02 \\ 
100 & 259 & 5.9e-03 & 1.0e-01\\ 
200 & 206 & 1.1e-02 & 1.4e-01\\ 
500 & 135 & 2.8e-02 & 2.1e-01\\ 
1000 & 94 & 5.8e-02 & 2.9e-01 \\ 
\end{tabular}
\end{minipage}
\\
\begin{minipage}{0.66\textwidth}
\includegraphics[width=\textwidth]{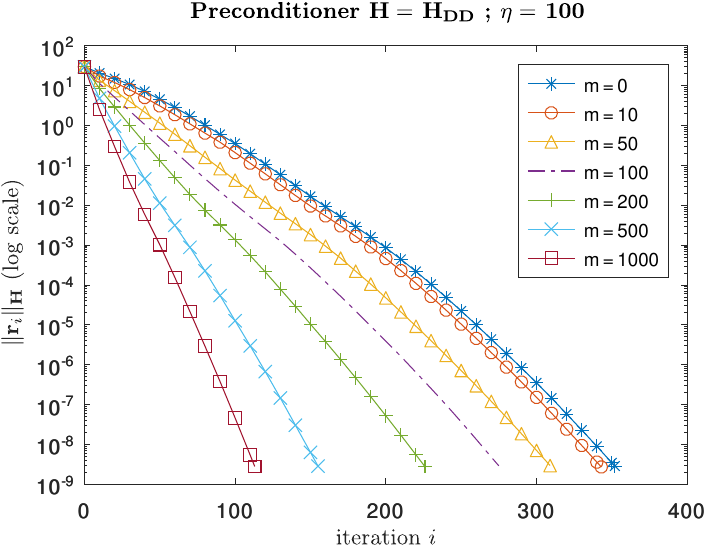}
\end{minipage}
\begin{minipage}{0.3\textwidth}
\begin{tabular}{rrcc}
m & iter &  $\theta_{th}$ & $\theta_{exp}$  \\ 
0 & 352 & 1.5e-05 & 7.55e-03 \\ 
10 & 343 & 6.0e-05 & 1.8e-02 \\ 
50 & 309 & 2.0e-04 & 4.4e-02 \\ 
100 & 275 & 3.6e-04 & 6.7e-02 \\ 
200 & 226 & 7.0e-04 & 1.0e-01 \\  
500 & 155 & 1.7e-03 & 1.9e-01 \\ 
1000 & 113 & 3.5e-03 & 2.9e-01 \\ 
\end{tabular}
\end{minipage}
\\
\begin{minipage}{0.66\textwidth}
\includegraphics[width=\textwidth]{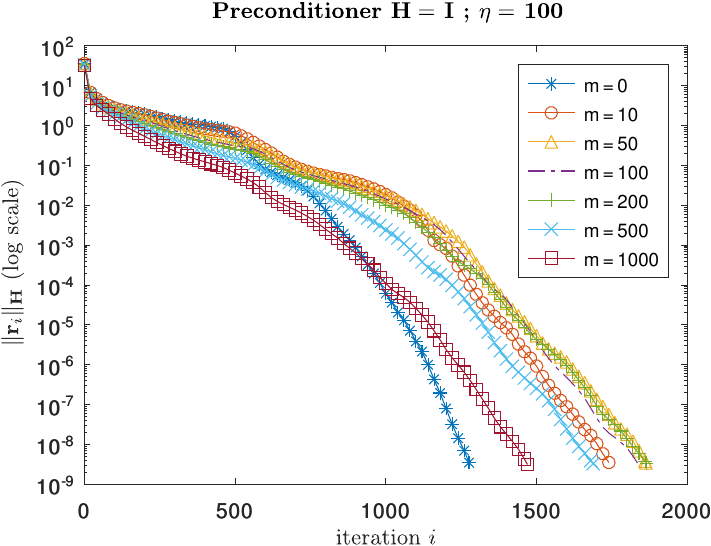}
\end{minipage}
\begin{minipage}{0.3\textwidth}
\begin{tabular}{rrcc}
m & iter &  $\theta_{th}$ & $\theta_{exp}$  \\ 
0 & 1276 & 2.4e-08 & 4.4e-03 \\ 
10 & 1740 & 9.9e-08 & 3.9e-03 \\ 
50 & 1863 & 3.2e-07 & 4.0e-03 \\ 
100 & 1835 & 6.0e-07 & 4.1e-03 \\ 
200 & 1865 & 1.2e-06 & 5.94e-03 \\ 
500 & 1688 & 2.8e-06 & 7.3e-03 \\ 
1000 & 1470 & 5.8e-06 & 9.2e-03 \\ 
\end{tabular}
\end{minipage}
\caption{Convergence curves when $\eta = 100$ in Section~\ref{sec:CDR} with varying preconditioner and rank of deflation space.}
\label{fig:eta100}
\end{figure}

\begin{figure}
\centering
\includegraphics[width=0.5\textwidth]{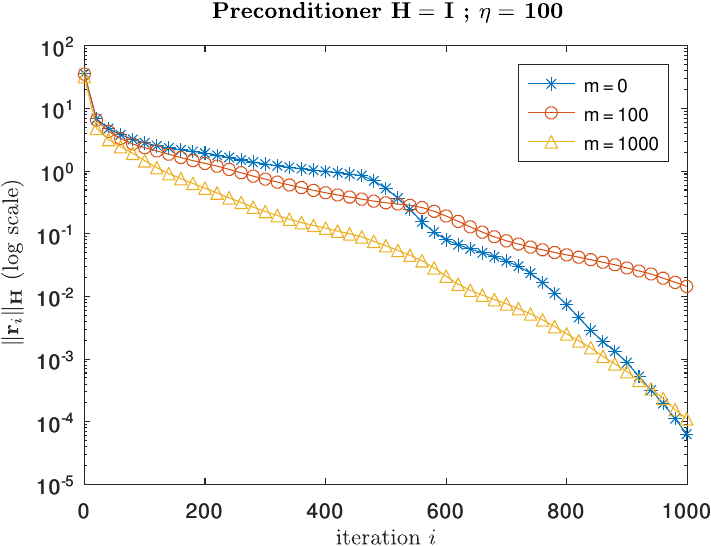}
\caption{Zoom on the first 1000 iterations in the case $\bH = \matid$ and $\eta = 100$. (Section~\ref{sec:CDR}).}
\label{fig:zoomprecIeta100}
\end{figure}

\paragraph{Comparison with the Elman bound}
A very well known bound for GMRES is the so-called  \textit{Elman bound} 
\cite{EES.83,elman1982iterative}. Its immediate generalization to GMRES preconditioned by $\bH$ and weighted by $\bW$ states that 
$
\frac{ \|\br_{i} \|_\bW^2}{\|\br_{i-1} \|_\bW^2} \leq  1 - \theta_\text{El}(\bA, \bH, \bW)
$
where
\[
\theta_\text{El}(\bA, \bH, \bW) = \left(\frac{\mathrm{d}(0,W_\bW(\bA \bH))}{\|\bA\bH\|_{\bW}}\right)^2.
\]
The numerator is the distance between $0$ and the $\bW$-field of value of the preconditioned operator.
When $\bW =\bH$ for some hpd matrix $\bH$, it can be shown that 
\[
\theta_\text{El}(\bA, \bH, \bH) = \frac{\lambda_{\min}(\bH \bM)^2}{\lambda_{\max}(\bA^* \bH \bA \bH)},
\]
Where we can use the $\lambda_{\max}$ notation since the eigenvalues are real and positive.
In Table~\ref{tab:Elman}, for each preconditioner and for $\eta \in \{1 , 100\}$ we compare $\theta_{th}$ and $\theta_{exp}$ to $\theta_\text{El}$. For the very particular case where $\bW = \bH = \bM^{-1}$, 
it can be proved that both bounds are identical.
For the other cases, it can be observed that our bound is lower than the Elman bound.

\begin{table}
\centering
\begin{tabular}{c|r|ccc}
 & & $\theta_{exp}$ & $\theta_{th}$ & $\theta_\text{El}$  \\
\hline
\multirow{2}{*}{$\bH = \Minv$} &$\eta = 1$ & 7.956e-01 & 7.059e-01 & 7.059e-01 \\
&$  \eta = 100$ & 3.144e-03 & 2.399e-04 & 2.399e-04 \\
\hline
\multirow{2}{*}{$\bH = \HDD$}& $\eta = 1$ &   5.314e-01 & 4.346e-02 &  2.399e-04 \\
& $  \eta = 100$ & 7.453e-03 & 1.477e-05 & 9.321e-06 \\
\hline
\multirow{2}{*}{$\bH = \matid$}& $\eta = 1$ &  1.658e-02 & 7.133e-05 & 1.021e-08 \\
& $  \eta = 100$ & 
4.447e-03 & 2.424e-08 & 7.774e-09
\end{tabular}
\caption{{For the convection-diffusion-reaction problem in Section~\ref{sec:CDR}  without deflation, comparison between $\theta_{exp} = \min_i (1-{\|\br_{i} \|_\bW^2}/{\|\br_{i-1} \|_\bW^2}$), $\theta_{th}$ the bound from this paper (also the bound in \cite{spillane2023hermitian}) and $\theta_{\text{El}}$, the Elman bound. Larger $\theta$ means faster convergence. 
\label{tab:Elman}}}
\end{table}

\paragraph{Influence of the mesh.}

We consider the case where $\eta = 100$ and $\bH = \HDD$ with  varying mesh size. The stopping criterion is
as before, $\| \br_i \|_\bH/ \| \br_0 \|_\bH = \| \br_i \|_\bH/ \| \rhs \|_\bH < 10^{-10} $. The three considered meshes are the $32158$ vertex mesh that was used in the tests above as well as the two less refined meshes represented in Figure~\ref{fig:meshnloc11} (2373 and 8643 vertices). After elimination of the degrees of freedom on the boundary, the resulting linear systems have respectively 31502, 8307 and 2197 dofs. The problem is solved by WPD-GMRES without deflation ($\text{m} = 0$ and with deflation of $\text{m} = 100$ vectors). The results are presented in Table~\ref{tab:hvaries}. We notice that the iteration counts increase weakly with the mesh size.
For example, from 236 for 2,197 dofs to 275 for 32,502 dofs in the case where 100 vectors are deflated. 
We also report on the quantities that appear in the convergence bounds: $\kappa(\HDD \bM)$, $| \lambda_{1} | = \rho(\bM^{-1}\bN)$ and $| \lambda_{101} |$. We know from the theory in \cite{spillane2013abstract} that $\kappa(\HDD \bM)$ is bounded independently of the mesh size and we indeed notice that it does not depend very much on the mesh. In \cite[Section 5]{spillane2023hermitian}, it was proved that for this particular PDE,  $| \lambda_{1} | = \rho(\bM^{-1}\bN)$ is also bounded independently of the mesh size by
\begin{equation}
\label{eq:boundrho}
\rho(\bM(\bA)^{-1} \bN(\bA)) \leq \frac{1}{2} \frac{\|\mathbf{a}\|_{L^\infty(\Omega)}}{\sqrt{\inf (\nu) \inf (c_0 + \frac{1}{2} \operatorname{div} (\mathbf a))}} = 3.23 \eta = 323. 
\end{equation}
The bound is satisfied here and $ \rho(\bM^{-1}\bN)$ does not depend on the mesh. For $| \lambda_{101} |$, we have no theoretical results (except of course that $| \lambda_{101} | \leq \rho(\bM(\bA)^{-1} \bN(\bA))$) and we observe a small increase when the number of dofs increases.   

\begin{table}
\centering
\begin{tabular}{r|cccccccc}
$\#$ dofs & \multicolumn{2}{c}{Iter} & $\kappa(\HDD \bM)$  &  $| \lambda_{1} | = \rho(\bM^{-1}\bN)$ &  $| \lambda_{101} |$ \\  
\hline
          &          $\text{m} = 0$                    &          $\text{m} = 100$                      & & & \\
2,197 & 316 & 236 & 14.4 & 64.4 & 11.5 \\ 
8,307 & 343 & 267 & 14.2 & 64.5 & 12.7 \\  
32,502 & 352 & 275 &13.0& 64.6 & 13.0 \\   
\end{tabular}
\caption{For $\eta = 100$ and $\bH = \HDD$ in Section~\ref{sec:CDR} , three mesh sizes are considered.
\label{tab:hvaries}}
\end{table}

\paragraph{Comparison with left preconditioned and deflated GMRES.} As a final test we examine whether GMRES with the same preconditioner (applied on the left) and deflation operator but without the change of inner product exhibits the same convergence behavior as WPD-GMRES (\textit{i.e.}, preconditioned and deflated GMRES in the $\bH$-inner product). To this end we solve the same problem
twice: once with WPD-GMRES and once with preconditioned and deflated GMRES. In both cases the stopping criterion is set to $\|\bH \br_i \|/ \|\bH \br_0 \| = \|\bH \br_i \|/ \|\bH \rhs \| < 10^{-10} $ where $\| \cdot \|$ is the Euclidean norm. As an illustration, we choose the problem with 8,307 dofs and set $\bH = \HDD$. Problems with ($\eta \in \{0.1, 1, 10, 100\}$) are solved without deflation and with deflation of $100$ vectors. The iteration counts can be found in Table~\ref{tab:vsGMRES}. 
The fact that the number of iterations is nearly identical in the weighted and unweighted cases is remarkable. The iteration count for the unweighted method is always the smallest. This is due to the fact that the chosen stopping criterion is precisely the norm that is minimized by the unweighted method. In Figure~\ref{fig:vsGMRES} we plot the history of residual $\|\bH \br_i \|/ \|\bH \br_0 \| $ for $\eta = 100$. Again, we observe a strong similarity and the minimization property of GMRES ensures that in this norm, the residual for $\bW = \matid$ will always be below the residual for $\bW = \bH$.  

\begin{table}
\centering
\begin{tabular}{r|rr|rr}
$\eta$ & \multicolumn{2}{c}{ $\text{m} = 0$  } &        \multicolumn{2}{c}{ $\text{m} = 100$} \\ 
          & $\bW = \matid$  & \bW = \bH    & $\bW = \matid$  & \bW = \bH \\ 
\hline
0.1 & 36 & 37 & 33 & 33 \\
1 & 40 & 40 & 33 & 34 \\
10 & 89 & 90 & 54 & 54 \\
100 & 327 & 329 & 253 & 255 \\
\end{tabular}
\caption{Iteration counts for weighted and unweighted algorithms. Both algorithms use the stopping criterion coming from the unweighted algorithm.
\label{tab:vsGMRES}}
\end{table}

\begin{figure}
\centering
\includegraphics[width = 0.6 \textwidth]{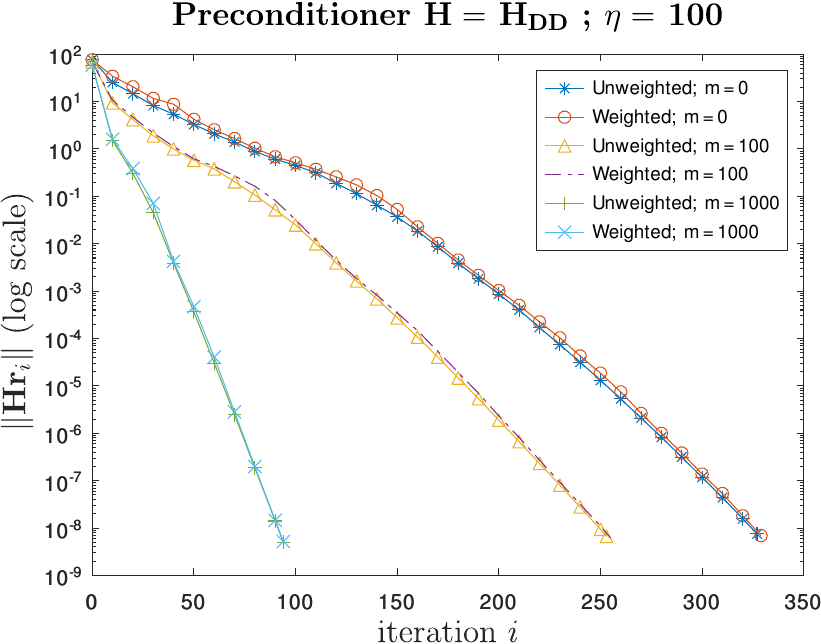}
\caption{Comparison of weighted and unweighted algorithms, without deflation and with deflation of $100$ vectors. The chosen norm is the one that is minimized by the unweighted algorithm.
\label{fig:vsGMRES}}
\end{figure}

\section{Conclusions}
We present a very general convergence bound for preconditioned GMRES when any inner product is used,
and when it is deflated. We consider bounds for several generic cases for the deflation space.
These bounds inspired us to produce an effective deflation space, namely, the 
eigenvectors of the generalized eigenvalue problem 
$\bN \bz = \lambda \bM \bz$, where $\bM$ is the Hermitian part of $\bA$, which is
assumed to be positive definite, and $\bN$ is the skew-Hermitian part of $\bA$.
Only eigenvectors corresponding to eigenvalues with moduli above a given threshold are selected.
Numerical experiments illustrate the potential for these ideas. 
In some cases, deflating is a key to convergence when GMRES without deflation fails to converge.
In other cases,
deflating indeed reduces the
number of iterations, and so do the preconditioners combined with the deflation.
On the other hand, while our theory applies to GMRES in the preconditioner inner product, our experiments show that the choice of inner product is not as important as the deflation space.

\section*{Acknowledgements} We thank the referees for 
making useful comments and suggestions, which helped improved our presentation.

\bibliographystyle{abbrv}
\bibliography{NonSymOrthodir}

\end{document}